\documentclass[reqno,english,11pt]{amsart}
\usepackage{amsfonts,amsmath,latexsym,verbatim,amscd,mathrsfs,color,array}
\usepackage{amssymb,amsthm,graphicx }
\usepackage[colorlinks=true]{hyperref}
\usepackage[normalem]{ulem}
\usepackage{enumitem}

\hypersetup{
	colorlinks=true,
	linkcolor=red
}
\usepackage{stmaryrd}

\allowdisplaybreaks

\usepackage[hmargin=2.3cm, vmargin=2.3cm]{geometry}
\usepackage{bbm} 
\usepackage{pdfsync}
\usepackage{epstopdf}
\usepackage{cite}
\usepackage{graphicx}
\usepackage{multirow}
\usepackage[font=bf,aboveskip=15pt]{caption}
\usepackage[toc,page]{appendix}

\usepackage{bookmark}
\DeclareFontShape{OT1}{cmr}{bx}{sc}{<-> cmbcsc10}{}
\usepackage[square,numbers]{natbib}

\usepackage{esint}
\usepackage{float}
\usepackage{ulem}
\usepackage{syntonly}
\usepackage{mathtools}
\usepackage{bm}

\usepackage{pdfsync}
\usepackage{epstopdf}
\usepackage{upgreek}

\newtheorem{theorem}{Theorem}
\newtheorem{proposition}{Proposition}[section]

\newtheorem{lemma}[proposition]{Lemma}

\theoremstyle{definition}

\numberwithin{equation}{section}

\newcommand{\R}{\mathbb{R}}
\newcommand{\N}{\mathbb{N}}
\newcommand{\Tr}{\operatorname{Tr}}
\newcommand{\rank}{\operatorname{rank}}
\newcommand{\dist}{\operatorname{dist}}
\newcommand{\ketbra}[1]{|#1\rangle\langle #1|}

\newcommand{\abs}[1]{\lvert #1 \rvert}

\newcommand{\Rank}{{\rm rank~ }}
\newcommand{\SQ}{{\mathcal{S}^q}}

\def\R{{\mathbb R}}

\def\inte{\int_{\mathbb{R}^3}}
\def\f{\frac}
\def\al{\alpha}
\def\ga{\gamma}
\def\va{\varepsilon}
\def\van{\varepsilon_n}

\def\l{\left}
\def\r{\right}
\def\la{\langle}
\def\ra{\rangle}

\usepackage{etoolbox}

\makeatletter

\patchcmd{\@settitle}
{\bfseries}
{\normalfont\fontsize{16}{18}\selectfont}
{}
{}

\patchcmd{\@settitle}
{\uppercasenonmath\@title}
{}
{}
{}

\patchcmd{\@setauthors}
{\MakeUppercase{\authors}}
{\authors}
{}
{}

\renewcommand\section{\@startsection{section}{1}%
	\z@{.7\linespacing\@plus\linespacing}{.5\linespacing}%
	{\normalfont\bfseries\raggedright}}

\makeatother

\begin{document}
	\title[Mass concentration of Fermi systems]
	{Mass concentration and ground states of Fermi systems in Hartree--Fock model with potentials}

	\date{}

        \author[Z. Li, Y. Wu and X. Zeng]{
	{\normalfont\large
		Zhuocheng Li$^{1}$ \qquad
		Yuan-da Wu$^{2}$ \qquad
		Xiaoyu Zeng$^{1}$}
	\\[9pt]
	{\normalfont\normalsize\itshape
		$^{1}$Center for Mathematical Sciences, Wuhan University of Technology,}
	\\[-1pt]
	{\normalfont\normalsize\itshape
		Wuhan 430070, P.R. China}
	\\[4pt]
	{\normalfont\normalsize\itshape
		$^{2}$School of Mathematical Sciences, Tianjin University of Technology,}
	\\[-1pt]
	{\normalfont\normalsize\itshape
		Tianjin 300384, P.R. China}
}
	
		\thanks{		E-mail: lzc2023203958@whut.edu.cn(Z.C. Li); wyuanda2021@126.com(Y. Wu); xyzeng@whut.edu.cn (X.Y. Zeng);.		}

	\maketitle
	
	

	\vskip0.36in

	\begin{abstract}
Understanding the onset and mechanism of collapse at critical attractive interaction is a fundamental problem in the variational theory of finite fermionic systems. In this paper, we study this problem within the Hartree--Fock framework under various trapping potentials. By establishing a finite-rank Hartree-type Gagliardo--Nirenberg--Sobolev inequality, we identify a critical threshold $\mathcal{K}^{(N)}$ governing the existence of ground states. We then determine the concentration behavior of ground states as the interaction parameter approaches $\mathcal{K}^{(N)}$. To reveal how the geometry of the trap influences this critical phenomenon, we consider three representative classes of potentials: polynomial-type potentials, ring-shaped potentials, and potentials whose zero set has nonempty interior. For each class, we derive the concentration scale and sharp asymptotics of the least energy. The resulting regimes are qualitatively different, showing that both the location and the rate of concentration are selected by the local geometry of the potential near its minimum set. In contrast to scalar Hartree problems, the nonlocal interaction is coupled here with pairwise $L^2$-orthogonality constraints, and their interplay creates substantial difficulties in the compactness and blow-up analysis.

	\end{abstract}
	\noindent {\bf MSC}: 35J20; 35J60; 35Q55
	
	\noindent {\bf Keywords}: Gagliardo-Nirenberg-Sobolev  inequality; finite Fermi system; ground state; asymptotic behavior

	\tableofcontents
	
	\vskip0.6in
	
	
	\section{Introduction}
	The mathematical analysis of interacting fermionic systems has attracted considerable attention in recent years, motivated by developments in trapped ultracold Fermi gases \cite{BrackMurthy,Dean,Giorgini,Smith}. Spinless fermions provide a convenient setting in which the effect of particle interactions can be studied without additional spin degrees of freedom. Among the effective descriptions of such systems, Hartree and Hartree--Fock theories lead naturally to variational problems with nonlocal interaction terms.
	
	In the present paper we consider finite spinless Fermi systems in three dimensions under external trapping potentials. In the Hartree--Fock--Slater framework, the attractive interaction is represented by a Hartree convolution term, and the energy functional is
	\begin{equation}
		E_K(\gamma)
		:=\Tr\bigl(( -\Delta+V(x))\gamma\bigr)
		-K\int_{\R^3}\bigl(\rho_\gamma*|x|^{-2}\bigr)\rho_\gamma\,dx,
		\label{eq:energy}
	\end{equation}
	where $K>0$ is the coupling strength. For a non-negative finite-rank self-adjoint operator
	\[
	\gamma=\sum_{j=1}^{N}n_j\ketbra{u_j},
	\qquad n_j\ge0,
	\qquad (u_i,u_j)_{L^2(\R^3)}=\delta_{ij},
	\]
	the associated density and kinetic energy are given by
	\[
	\rho_\gamma(x)=\sum_{j=1}^{N}n_j|u_j(x)|^2,
	\qquad
	\Tr(-\Delta\gamma)=\sum_{j=1}^{N}n_j\int_{\R^3}|\nabla u_j|^2\,dx.
	\]
	The nonlocal Hartree term in \eqref{eq:energy} is critical with respect to the same spatial scaling as the kinetic energy: under the dilation $u_\lambda(x)=\lambda^{3/2}u(\lambda x)$, both $\|\nabla u_\lambda\|_2^2$ and $\int (\rho_\lambda\ast |x|^{-2})\rho_\lambda\,dx$ scale as $\lambda^2$, and this critical balance is responsible for the threshold and concentration phenomena considered below.
	
	Let
	\[
	\mathcal H
	:=\left\{u\in H^1(\R^3):\int_{\R^3}V(x)|u(x)|^2\,dx<\infty\right\}
	\]
	and, for $N\in\N^+$, define
	\[
	\Upsilon_N
	:=\left\{\gamma=\sum_{j=1}^{m}\ketbra{u_j}:\ 1\le m\le N,
	\ u_j\in\mathcal H,\ (u_i,u_j)_{L^2}=\delta_{ij}\right\}.
	\]
	Our main variational problem is
	\begin{equation}
		e_K(N):=\inf\{E_K(\gamma):\gamma\in\Upsilon_N\}.
		\label{eq:mainvar}
	\end{equation}
	We assume throughout the paper that the trapping potential satisfies
	\begin{equation*}
		\tag{V1}\label{V1}
		0\le V\in C(\R^3),
		\qquad \lim_{|x|\to\infty}V(x)=+\infty,
		\qquad \inf_{x\in\R^3}V(x)=0.
	\end{equation*}
	Under our assumptions on $V$, the embedding $\mathcal H\hookrightarrow L^2(\mathbb{R}^3)$ is compact, which will be crucial for recovering compactness of minimizing sequences. And the condition $V(x)\to+\infty$ makes \eqref{eq:mainvar} a natural model for confined fermionic states.
	
	Related ground-state problems for finite Fermi systems with local nonlinear interactions have been studied in several settings. A pioneering result in this field was obtained by Guo and Li \cite{GuoLi}, who considered two-spinless mass-critical Fermi systems under trapping potentials and identified a sharp critical interaction strength separating the existence and non-existence of ground states. Subsequently, for Coulomb-type external potentials, Chen and Guo \cite{ChenGuo} established the existence of ground states in the $L^2$-subcritical regime. These results naturally call for a rigorous treatment of the existence problem for finite Fermi systems with a nonlocal Hartree interaction.
	
	A fundamental question is whether a minimizer of \eqref{eq:mainvar} exists, and how this depends on the parameter $K$. The attractive Hartree term shows that this question is governed by the sharp constant
	\begin{equation}
		0<\mathcal{K}^{(N)}
		:=\inf_{\substack{0<\gamma=\gamma^*\le1\\1\le\rank\gamma\le N}}
		\frac{\|\gamma\|_{\mathcal{S}^\infty}\Tr(-\Delta\gamma)}
		{\displaystyle\int_{\R^3}(\rho_\gamma*|x|^{-2})\rho_\gamma\,dx},
		\label{eq:critical}
	\end{equation}
	where for $\gamma=\sum_{i=1}^N n_i|u_i\ra\la u_i|$, define
	\begin{equation*}
		\|\gamma\|_{\mathcal{S}^q}:=\begin{cases}
			(\sum_{i=1}^{\infty}n_i^q)^{\frac{1}{q}}<\infty ,\quad &\text{ if } 1\le q<\infty,\\
			\max_{i}n_i<\infty,\quad &\text{ if } q=\infty.
		\end{cases}
	\end{equation*}
	Indeed, \eqref{eq:critical} is exactly the best constant in the estimate controlling the Hartree interaction by the kinetic energy, and therefore provides the natural candidate for the critical coupling in \eqref{eq:mainvar}.
	
	If $\gamma=\sum_{i=1}^{m}\ketbra{u_i}$ is a minimizer, the functions $\{u_i\}$ satisfy the self-consistent system
	\begin{equation}
		\left[-\Delta+V(x)-2K\bigl(\rho_\gamma*|x|^{-2}\bigr)\right]u_i
		=\mu_i u_i
		\qquad\text{in }\R^3,
		\label{eq:EL}
	\end{equation}
	where the occupied states $\{u_i\}$ correspond to the lowest eigenvalues of the effective operator. We use the term \textbf{ground state} for such a minimizing orthonormal family.
	
	Our first result identifies the sharp existence threshold for problem \eqref{eq:mainvar}.
	
	\begin{theorem}\label{thm:existence}
		Let $N\in\N^+$, let $\mathcal{K}^{(N)}$ be defined by \eqref{eq:critical}, and assume that $V$ satisfies \eqref{V1}.
		\begin{enumerate}[label=\textup{(\arabic*)},leftmargin=2.2em]
			\item If $0<K<\mathcal{K}^{(N)}$, then $e_K(N)$ admits at least one minimizer
			\[
			\gamma=\sum_{i=1}^{m}\ketbra{u_i},\qquad 1\le m\le N,
			\]
			and $(u_1,\cdots,u_m)$ is a ground state of \eqref{eq:EL}.
			\item If $K>0$ is sufficiently small, the minimizer $\gamma$ is unique and $\rank \gamma=m=1$. If $K<\mathcal{K}^{(N)}$ is sufficiently close to $\mathcal{K}^{(N)}$,  then $\rank \gamma=m\geq \lfloor N/2 \rfloor+1$.
			\item If $K\ge \mathcal{K}^{(N)}$, then $e_K(N)$ has no minimizer.
		\end{enumerate}
	\end{theorem}
	
	Theorem~\ref{thm:existence} shows that the sharp constant \eqref{eq:critical} is not merely an auxiliary quantity: it is the exact threshold separating the existence and non-existence of confined ground states. In the subcritical regime, the proof proceeds by combining the coercivity from \eqref{eq:critical} with the compactness induced by the trapping potential, which leads to the existence result. At and above the critical value, the principal difficulty is to construct trial states which preserve orthogonality while retaining the same critical scaling for the kinetic and Hartree terms. In addition, the rank statement near $K=0$ is perturbative, whereas the rank behavior near $\mathcal{K}^{(N)}$ is governed by strict inequalities among the finite-rank best constants as functions of $N$-a phenomenon that will be made precise through a sharp finite-rank Gagliardo-Nirenberg-Sobolev(GNS) inequality, which we now introduce.
	
	This motivates us to establish the following Hartree-type analogue of the finite-rank GNS inequality, which will be the main tool in the proof of Theorem \ref{thm:existence}. Closely related local problems have been studied for nonlinear Schr\"odinger systems under orthogonality constraints. Gontier, Lewin and Nazar \cite{Lewin} considered the variational problem
	\[
	J(N)=\inf\left\{\Tr(-\Delta\gamma)-\frac1p\int_{\R^d}\rho_\gamma^p\,dx:\
	\gamma=\sum_{i=1}^{N}\ketbra{u_i},\ (u_i,u_j)=\delta_{ij}\right\}
	\]
	in the mass-subcritical regime. Using concentration--compactness principle and binding inequalities, they proved the existence of minimizers and obtained the associated orthonormal nonlinear Schr\"odinger system. Their work highlights the importance of GNS inequalities in studying fermionic ground-state problems.
	
	Frank, Gontier and Lewin \cite{FGL} subsequently developed the corresponding finite-rank GNS inequality. In their local setting, the interaction is expressed through an $L^p$-norm of $\rho_\gamma$. The present problem requires an analogue in which the local nonlinearity is replaced by the nonlocal Hartree interaction. Nonlocal variational structures of Hartree or Hartree--Fock type have also appeared in the analysis of fermionic systems; see, for instance, \cite{ChenGuoNamOuYang,LenzmannLewin,WuZengZhang}. These works further motivate the study of sharp finite-rank inequalities in a Hartree setting.
	
	Moreover, motivated by the duality between GNS and Lieb--Thirring
	inequalities in related finite-rank problems \cite{FGL,WuZengZhang},
	the Hartree GNS inequality established here also leads to a natural
	dual Hartree-type Lieb--Thirring inequality. Although this inequality
	is not needed in the present work, we include it in Appendix~B as a complement.
	
	For this purpose, we introduce a more general family of best constants. For $0<\alpha<3$ and $1\le q\le\infty$, define
	\begin{equation}
		\mathcal{K}_{\alpha,q}^{(N)}
		:=\inf_{\substack{0<\gamma=\gamma^*\\1\le\rank\gamma\le N}}
		\frac{\|\gamma\|_{\mathcal{S}^q}^{\frac{4-\alpha}{\alpha}}\Tr(-\Delta\gamma)}
		{\left(\displaystyle\int_{\R^3}(\rho_\gamma*|x|^{-\alpha})\rho_\gamma\,dx\right)^{\frac2\alpha}}.
		\label{eq:GNS}
	\end{equation}
	The constant in \eqref{eq:critical} is precisely the special case
	\[
	\mathcal{K}^{(N)}=\mathcal{K}_{2,\infty}^{(N)}.
	\]
	The following theorem provides a complete characterization of all finite-rank best constants.
	
	Set
	\[
	q_\alpha^*:=\frac{4-\alpha}{(2-\alpha)_+}
	=\begin{cases}
		\dfrac{4-\alpha}{2-\alpha},&0<\alpha<2,\\[0.8ex]
		+\infty,&2\le\alpha<3.
	\end{cases}
	\]
	
	\begin{theorem}\label{thm:GNS}
		Let $1\le N<\infty$, $0<\alpha<3$, and $1\le q\le q_\alpha^*$. Then the infimum $\mathcal{K}_{\alpha,q}^{(N)}$ in \eqref{eq:GNS} is attained by a positive operator $\gamma$ with $1\le\rank\gamma\le N$. Moreover, up to the natural scaling:
		\begin{enumerate}[label=\textup{(\roman*)},leftmargin=2.2em]
			\item If $1\le q<q_\alpha^*$, every minimizing sequence for \eqref{eq:GNS} is compact up to scaling.
			\item Define
			\[
			H_{\alpha,\gamma}
			:=-\Delta-\frac4\alpha\bigl(\rho_\gamma*|x|^{-\alpha}\bigr).
			\]
			An optimizer can be written as
			\[
			\gamma=\sum_{i=1}^{M}k_i\ketbra{u_{n_i}},\qquad
			k_i>0,\quad 1\le M\le N,
			\]
			where $(\mu_{n_i},u_{n_i})$ are eigenpairs of $H_{\alpha,\gamma}$ with $\mu_{n_i}\le0$. If $M<N$, then $H_{\alpha,\gamma}$ has at most $M$ negative eigenvalues.
			\item If $1<q<q_\alpha^*$, then $\mu_{n_i}<0$ and
			\[
			k_i=\frac{4-\alpha}{\alpha}
			\left(\sum_{k=1}^{M}|\mu_{n_k}|^{\frac{q}{q-1}}\right)^{-1}
			|\mu_{n_i}|^{\frac1{q-1}}.
			\]
			For $\alpha=2$ and $q=\infty$,
			\[
			k_i=\left(\sum_{k=1}^{M}|\mu_{n_k}|\right)^{-1}.
			\]
			Especially, $\mu_{n_i}$ are the first $M$ negative eigenvalues of $H_{\alpha,\gamma}$ provided that either $0<\alpha\leq2$, or $2<\alpha<3$ and rank $\gamma < N$.
			\item If $0<\alpha\le2$ and $q$ is sufficiently close to $q_\alpha^*$, there exists a strict increasing sequence $\{N_j\}$ such that
			\[
			\mathcal{K}_{\alpha,q}^{(N_j)}<\mathcal{K}_{\alpha,q}^{(N_j-1)}.
			\]
			In particular,
			\begin{equation}
				\mathcal{K}_{\alpha,q}^{(2N)}<\mathcal{K}_{\alpha,q}^{(N)},\qquad N\in\N^+.
				\label{2n<n}
			\end{equation}
		\end{enumerate}
	\end{theorem}
	
	The main new feature of Theorem~\ref{thm:GNS} is the nonlocal Hartree interaction. In contrast with a local $L^p$ term, a spatial decomposition of $\rho_\gamma$ do not interact, the nonlocal Hartree term retains cross-interactions between them through the convolution kernel. Consequently, the compactness argument has to control both the localization of the kinetic energy and the decay of the Riesz cross term. This theorem is one of the key analytical ingredients in the proof of Theorem~\ref{thm:existence}.

	Having identified the sharp existence threshold, we return to the main variational problem and ask what happens to ground states as
	\[
	K\nearrow \mathcal{K}^{(N)}.
	\]
	This is the second main theme of the paper. In the local mass-critical Fermi model, Guo and Li \cite{GuoLi} proved that ground states exist exactly below a critical interaction strength and, as the critical value is approached, concentrate at the flattest minimum points of the trapping potential. Their analysis shows that the critical existence threshold and the blow-up behavior of minimizing states should be viewed as two parts of the same variational problem.
	
	Our next theorem gives the general critical concentration behavior under \eqref{V1}.
	
	\begin{theorem}\label{thm:concentration}
		Assume that $V$ satisfies \eqref{V1}, and let
		\[
		\gamma_n=\sum_{i=1}^{m}|u_i^n\rangle\langle u_i^n|
		\]
		be a minimizer of $e_{K_n}(N)$, where
		$\lfloor N/2 \rfloor+1\leq m\leq N$, $K_n\nearrow \mathcal{K}^{(N)}$, and
		$\{u_i^n\}_{i=1}^{m}\subset\mathcal H$ satisfies \eqref{eq:EL}.
		Set
		\[
		\varepsilon_n^{-2}:=\Tr(-\Delta\gamma_n).
		\]
		Then the following statements hold.
		
		\begin{enumerate}[label=\textup{(\roman*)},leftmargin=2.2em]
			\item
			\begin{equation}
				\varepsilon_n\rightarrow 0
				\qquad\text{as }K_n\nearrow \mathcal{K}^{(N)}.
				\label{3:e}
			\end{equation}
			
			\item
			There exist a sequence $\{y_{\varepsilon_n}\}\subset\mathbb R^3$ such that, defining
			\begin{equation}
				\bar w_i^n(x)
				:=\varepsilon_n^{3/2}
				u_i^n(\varepsilon_n x+\varepsilon_n y_{\varepsilon_n}),
				\qquad
				\bar\gamma_n:=\sum_{i=1}^{m}
				|\bar w_i^n\rangle\langle\bar w_i^n|,
				\qquad
				\rho_{\bar\gamma_n}:=\sum_{i=1}^{m}|\bar w_i^n|^2,
				\label{3:2}
			\end{equation}
			and $\bar\gamma_n$ be a optimizer of (\ref{eq:mainvar}) for each $K_n\nearrow \mathcal{K}^{(N)}$, we have
			\[
			\Tr(-\Delta\bar\gamma_n)=1,
			\qquad
			\int_{\mathbb R^3}
			(\rho_{\bar\gamma_n}*|x|^{-2})\rho_{\bar\gamma_n}\,dx
			\rightarrow \frac{1}{\mathcal{K}^{(N)}}, \qquad \bar\gamma_{n}\overset{\star}\rightharpoonup \bar\gamma \quad \text{in}\ \mathcal{S}^{1},
			\]
			where $\bar\ga:=\sum_{i=1}^M|\bar w_i\ra\la \bar w_i| (M\leq m)$ is a optimizer of $\mathcal{K}^{(N)}$ defined by \eqref{eq:critical}.
			
			\item
			Let
			\[
			\bar x_n:=\varepsilon_n y_{\varepsilon_n},
			\qquad
			\mathcal A:=\{x\in\mathbb R^3:V(x)=0\}.
			\]
			Then
			\begin{equation}
				\lim_{K_n\nearrow \mathcal{K}^{(N)}}
				\dist(\bar x_n,\mathcal A)=0.
				\label{03-1}
			\end{equation}
			Moreover, when $M=m$, we get $\bar\ga_n\rightarrow \bar\ga$ in $\mathcal{S}^1\cap \mathcal{S}^\infty$, and up to a subsequence, there exists a point $x_0\in\mathcal{A}$ such that
			\begin{equation}\label{3-1}
				\bar x_n\stackrel{n}{\longrightarrow}
				x_0\ \ \hbox{and}\ \ \bar w_i^{n}(x):=\va_{n}^{\f32}u_i^{n}(\va_{n} x+\bar x_n)\stackrel{n}{\longrightarrow} \bar w_i(x)
			\end{equation}
			strongly in $H^1(\R^3)$, for all $i=1, \cdots, M$.
		\end{enumerate}
	\end{theorem}
	
	Theorem~\ref{thm:concentration} gives a general description of the critical profile: the kinetic energy diverges, the rescaled states converge to an optimizer of the sharp inequality, and concentration takes place near the zero set of the trap. It does not yet determine the precise rate of concentration or, when the zero set is degenerate, which part of it is selected. These finer questions depend on the geometry of $V$.
	
	The influence of the trapping geometry on critical concentration has been studied extensively in related nonlinear problems. For attractive Bose--Einstein condensates, Guo and Seiringer \cite{GuoSeiringer} analyzed mass concentration near criticality. Guo, Zeng and Zhou \cite{GuoZengZhou} treated ring-shaped trapping potentials, while Guo, Wang, Zeng and Zhou \cite{GuoWangZengZhou} studied multi-well potentials. Together with the fermionic concentration result of Guo and Li \cite{GuoLi}, these works suggest that the next-order asymptotics should distinguish isolated polynomial minima, ring-shaped minima, and traps whose zero set has nonempty interior.
	
	Motivated by those paper, we impose three more specific assumptions on $V$. 
	
	(I) We first focus on potentials of polynomial type, and take 
	\begin{equation*}
		\tag{V2}\label{V2}
		V(x)=g(x)\prod_{j=1}^{\ell}|x-x_j|^{p_j},
		\qquad C^{-1}<g(x)<C,
	\end{equation*}
	where $x_j\ne x_k$ for $j\ne k$, $g\in C_{\rm loc}^{\kappa}(\R^3)$ for some $\kappa\in(0,1)$, and the limits of $g$ at the $x_j$ exist. Put
	\[
	p:=\max_{1\le j\le\ell}p_j,
	\qquad
	\mathcal Z:=\{x\in\R^3:V(x)=0\}=\{x_1,\cdots,x_\ell\},
	\]
	and
	\begin{equation}
		\Lambda_1:=\{x_j\in\mathcal Z:\alpha_j=\alpha_0\},
		\qquad
		\alpha_0:=\min_{1\le j\le\ell}\alpha_j,
		\qquad
		\alpha_j:=\lim_{x\to x_j}\frac{V(x)}{|x-x_j|^p}.
		\label{eq:lambda1}
	\end{equation}
	Thus $\Lambda_1$ contains the effective flattest minima at the blow-up scale.
	
	(II) We next restrict our attention to the so-called ring-shaped potential, namely
	\begin{equation*}
		\tag{V3}\label{V3}
		V(x)=(|x|-A)^2,\qquad A>0,
	\end{equation*}
	whose zero set is
	\[
	\Lambda_2:=\{x\in\R^3:|x|=A\}.
	\]
	
	(III) We finally study the completely degenerate case, and assume that
	\begin{equation*}
		\tag{V4}\label{V4}
		\overline\Omega=\{x\in\R^3:V(x)=0\},
		\qquad
		\Omega=\bigcup_{j=1}^{m_0}\Omega_j,
	\end{equation*}
	where $\Omega_1,\cdots,\Omega_{m_0}$ are mutually disjoint bounded domains and $V>0$ on $\R^3\setminus\overline\Omega$. Denote
	\begin{equation*}
		R_j:=\max_{x\in\Omega_j}\dist(x,\partial\Omega_j),
		\qquad
		R:=\max_{1\le j\le m_0}R_j,
	\end{equation*}
	and let $\mu_1<\cdots\le\mu_M<0$ be the occupied negative eigenvalues of the limiting effective operator. The relevant components satisfy the geometric condition
	\[
	\Lambda_3
	:=\Big\{\Omega_j:\ R_j\ge\sqrt{\frac{\mu_M}{\mu_1}}\,R\Big\}.
	\]
	The assumptions \eqref{V2}--\eqref{V4} correspond to different geometries of the zero set of V. For \eqref{V2}, the zeros are isolated and the main issue is to determine the selected minimum through the local flatness of V. For \eqref{V3}, the zero set is a sphere, so a finer expansion is needed to identify the concentration scale. For \eqref{V4}, the zero set has nonempty interior, and the asymptotics are governed instead by the decay of the limiting profile and the distance to the boundary. Accordingly, the asymptotic behavior of minimizers depends essentially on the geometry of the zero set of V.

	Under these assumptions, the general concentration statement can be sharpened to precise energy and spatial asymptotics.
	
	\begin{theorem}\label{thm:rates}
		Let $K_n\nearrow \mathcal{K}^{(N)}$ and assume the hypotheses of Theorem~\ref{thm:concentration}. Under any of \eqref{V2}--\eqref{V4}, for different $x_0$, $\bar w_i^{n}(x)\rightarrow \bar w_i(x)$ in $H^1(\R^3)$ to an optimizer of $\mathcal{K}^{(N)}$. More precisely:
		\begin{enumerate}[label=\textup{(\roman*)},leftmargin=2.2em]
			\item Suppose that $V$ satisfies \eqref{V2}. Then
			\begin{equation*}
				e_{K_n}(N)
				=(1+o(1))\frac{p+2}{p\mathcal{K}^{(N)}}
				M_1^{-\frac{2}{p+2}}
				\bigl(\mathcal{K}^{(N)}-K_n\bigr)^{\frac{p}{p+2}}
			\end{equation*}
			as $K_n\nearrow \mathcal{K}^{(N)}$, where $\alpha_0>0$ is defined in \eqref{eq:lambda1} and
			$$M_1=\frac{2 }{ \al_0 p\mathcal{K}^{(N)}\inte |x+\bar x|^{p}\rho_\ga dx }.$$
			Moreover,
			\[
			\varepsilon_n
			=\Tr(-\Delta\gamma_n)^{-1/2}
			=(1+o(1))\left[M_1\bigl(\mathcal{K}^{(N)}-K_n\bigr)\right]^{\frac1{p+2}},
			\]
			
			Furthermore, after passing to a subsequence, there exist
		$x_0\in\Lambda_1$ and $\bar x\in\mathbb R^3$ such that
		\[
		\frac{\bar x_n-x_0}{\varepsilon_n}\to\bar x,
		\qquad
		\bar w_i^n(x)
		:=\varepsilon_n^{3/2}
		u_i^n(\varepsilon_n x+\bar x_n)
		\to\bar w_i(x)
		\quad\text{strongly in }H^1(\mathbb R^3),
		\]
		here and in what follows, $\gamma_n$, $\bar x_n$, $\bar w_i^n$, and $\bar w_i$
		are as introduced in Theorem~\ref{thm:concentration}.
			\item Suppose that $V$ satisfies \eqref{V3}, 
			without loss of generality, after a rotation, assume that 
			\[
			x_0=x_A:=(A,0,0)\in\Lambda_2.
			\]
			Writing $x=(x_1,x_2,x_3)\in\mathbb R^3$, define
			\[
			M_2
			:=
			\int_{\mathbb R^3}x_1^2\rho_{\bar\gamma}(x)\,dx
			-
			\frac{
				\left(\int_{\mathbb R^3}x_1\rho_{\bar\gamma}(x)\,dx\right)^2
			}{
				\int_{\mathbb R^3}\rho_{\bar\gamma}(x)\,dx
			}.
			\]
			Then
			\[
			e_{K_n}(N)
			=(2+o(1))\sqrt{\frac{M_2}{\mathcal{K}^{(N)}}}\bigl(\mathcal{K}^{(N)}-K_n\bigr)^{1/2}
			,
			\]
			and
			\[
			\varepsilon_n
			=(1+o(1))\left(\frac{\mathcal{K}^{(N)}-K_n}{\mathcal{K}^{(N)}M_2}\right)^{1/4}.
			\]
			Furthermore, for some  $\bar x\in\R^3$,
			\[
			\frac{\bar x_{n}-x_A}{\varepsilon_n}\rightarrow\bar x,
			\qquad
			\bar{w}_i^n(x):=\varepsilon_n^{3/2}u_i^n(\varepsilon_n x+\bar x_{n})
			\rightarrow \bar{w}_i(x)
			\quad\text{strongly in }H^1(\R^3).
			\]
			\item Suppose that $V$ satisfies \eqref{V4}. Then, as $K_n\nearrow \mathcal{K}^{(N)}$,
			\[
			\frac{1+o(1)}{4|\mu_1|R_0^2\mathcal{K}^{(N)}}
			\bigl(\mathcal{K}^{(N)}-K_n\bigr)|\ln(\mathcal{K}^{(N)}-K_n)|^2
			\le e_{K_n}(N)
			\]
			\[
			\le
			\frac{1+o(1)}{4|\mu_M|R^2\mathcal{K}^{(N)}}
			\bigl(\mathcal{K}^{(N)}-K_n\bigr)|\ln(\mathcal{K}^{(N)}-K_n)|^2.
			\]
			
			More precisely, 
			after passing to a subsequence, there exist a component
			$\Omega_{j_0}\in\Lambda_3$, a point
			$x_0\in\overline{\Omega}_{j_0}$, and concentration points
			$\bar x_{n}$ such that
			\[
			\bar x_{n}\rightarrow x_0 ,
			\]
			and
			\begin{equation}
				R_0
				:=
				\lim_{n\to\infty}
				\dist(\bar x_{n},\partial\Omega_{j_0})
				\geq
				\sqrt{\frac{\mu_M}{\mu_1}}\,R .
				\label{1:R0}
			\end{equation}
			
			Moreover, 
			\[
			\varepsilon_n
			:=
			\Tr(-\Delta\gamma_n)^{-1/2}
			=
			(1+o(1))(C^*
			\bigl|\ln(\mathcal{K}^{(N)}-K_n)\bigr|)^{-1}
			,
			\] 
			and for some $\bar x\in\mathbb R^3$,
			\[
			\frac{\bar x_{n}-x_0}{\varepsilon_n}
			\rightarrow \bar x,\qquad
			\text{and}\qquad
			\bar{w}_i^n(x)
			:=
			\varepsilon_n^{3/2}
			u_i^n(\varepsilon_n x+\bar x_{n})
			\rightarrow
			\bar w_i(x)
			\]
			strongly in $H^1(\mathbb R^3)$ for every $i=1,\cdots,M$, 
			where
			$
			C^*
			\in
			\big(
			\frac{1}
			{2\sqrt{|\mu_1|}\,R_0},
			\frac{1}
			{2\sqrt{|\mu_M|}\,R}
			\big) .
			$
		\end{enumerate}
	\end{theorem}

	Theorem \ref{thm:rates} shows that the geometry of the zero set of V gives rise to different concentration scales. The main difficulty in deriving the refined asymptotics is to strengthen the general concentration result to strong convergence of the rescaled minimizers, while simultaneously determining the concentration points and the precise blow-up scales. The analysis depends essentially on the structure of the zero set of V: isolated minima can be treated by local expansions of V; the spherical case requires a finer analysis due to the degeneracy of the minimum set; whereas, when V vanishes on bounded domains, the leading-order behavior is determined by the exponential decay of the limiting profiles and their distance to the boundary. More recently, Guo, Li and Wu \cite{GLW} obtained refined blow-up results for a local Fermi model with a ring-shaped trapping potential. In contrast, the present problem involves a nonlocal Hartree interaction, for which the convolution structure produces additional interaction terms under localization and makes the blow-up analysis substantially more delicate. Moreover, we treat several different trapping geometries, including the multi-well case, which, to the best of our knowledge, has not been considered previously in this nonlocal fermionic setting.

	The main purpose of this paper is to study the existence and asymptotic behavior of minimizers for \eqref{eq:mainvar}. The finite-rank GNS problem in Theorem~\ref{thm:GNS} arises as a key ingredient in the proof of Theorem~\ref{thm:existence}, while Theorems~\ref{thm:concentration} and \ref{thm:rates} describe the concentration of minimizers as $K\nearrow \mathcal{K}^{(N)}$ under general and more specific assumptions on $V$, respectively.
	
	The paper is organized as follows. Section~\ref{sec:2} is devoted to the finite-rank Hartree GNS inequality. Section~\ref{sec:3} proves the existence, rank and non-existence result for $e_K(N)$. Section~\ref{sec:4} establishes the general critical concentration theorem.  Section~\ref{sec:5} derives the refined energy estimates and concentration behavior under \eqref{V2}--\eqref{V4}. The regularity and decay properties of the limiting optimizers, which are needed in the paper, are established in Appendix~\ref{sec:A}. For completeness, a dual Hartree-type Lieb--Thirring inequality is also included in Appendix~\ref{sec:B}.

	\section{Finite rank Hartree Gagliardo--Nirenberg--Sobolev inequality}\label{sec:2}
	In this section, we intend to establish the finite-rank Hartree
	Gagliardo--Nirenberg--Sobolev inequality stated in Theorem \ref{thm:GNS}, which use concentration-compactness principle to finish the proof.
	The main point is to prove the existence and compactness of optimizers
	in the presence of the nonlocal Hartree interaction.
	
	\begin{proof}[Proof of Theorem \ref{thm:GNS}]
		At first, we will prove that there exist a minimizer of \eqref{eq:GNS}. Assume there exists a minimizing sequence $\{\gamma_n\}$, define 
		$$ \gamma_n:=\sum_{i=1}^N a_{in}|u_i\ra\la u_i|,\quad a_{in}\geq0, \quad\mathrm{rank}(\ga_n)\leq N \ \mbox{for all} \ n,$$
		where $a_{in}=0$ when $i>\mathrm{rank}(\ga_n)$, which provided that $\mathrm{rank}(\ga_n)\leq N$.
		Normalizing the sequence such that
		\begin{equation}\label{norma1}
			\Tr(-\triangle\gamma_n)=1,\quad \|\gamma_n\|_{\mathcal{S}^q}=1.
		\end{equation}
		Then, denote $\rho_n:=\rho_{\gamma_n}$,
		\begin{equation*}
			\int_{\R^3}\int_{\R^3}\dfrac{\rho_n(x)\rho_n(y)}{|x-y|^\alpha}dxdy\rightarrow(\mathcal{K}^{(N)}_{\alpha,q})^{-\frac{\alpha}{2}}.
		\end{equation*}
		It follow from $\|\gamma_n\|\leq\|\gamma_n\|_{\mathcal{S}^q}=1$ that
		\begin{equation*}
			\int_{\R^3}\rho_n(x)dx=\|\gamma\|_{S^1}=\sum_{i=1}^N a_{in}\leq N.
		\end{equation*}
		In addition, we get from Hoffmann-Ostenhof inequality,
		\begin{equation*}
			1=\Tr(-\triangle\gamma_n)\geq \int_{\R^3}|\nabla\sqrt{\rho_n(x)}|^2dx.
		\end{equation*}
		Combining the above estimates, $\{\sqrt{\rho_n}\}$ is bounded in $H^1(\R^3)$, then there exist a subsequence, still denote $\{\sqrt{\rho_n}\}$,
		$$\sqrt{\rho_n}\rightharpoonup \sqrt{\rho}\quad \mbox{in}\quad H^1(\R^3),\  \ \mbox{and}\ \ \sqrt{\rho_n}\rightharpoonup \sqrt{\rho} \quad \mbox{in}\quad L^p(\R^3), \quad\mbox{where}\quad 2\leq p<6.$$
		If $\rho=0$, by vanishing lemma, we have $\sqrt{\rho_n}\rightarrow \sqrt{\rho}\equiv0 \ \mbox{in}\ L^p(\R^3), \ 2< p<6$.
		$$(\mathcal{K}^{(N)}_{\alpha,q})^{-\frac{\alpha}{2}}\leftarrow \int_{\R^3}\int_{\R^3}\dfrac{\rho_n(x)\rho_n(y)}{|x-y|^\alpha}dxdy\leq C\|\rho\|_{\frac{6-\alpha}{6}}^2 \rightarrow0,$$
		which is a contradiction, therefore $\rho\neq 0$. Applying Lion's method, there exists $\{R_n\}$ satisfying $R_n\rightarrow\infty\  \mbox{as}\ n\rightarrow\infty$, so that
		\begin{equation}\label{A3}
			\lim_{n\rightarrow\infty}\int_{|x|\leq R_n}\rho_n(x)dx=\int_{\R^3}\rho(x)dx, \quad \mbox{and}\quad\lim_{n\rightarrow\infty}\int_{R_n\leq|x|\leq 6R_n}\rho_n(x)dx=0.
		\end{equation}
		Set cutting function $\chi(x)\in C_c^\infty(\R^3,[0,1])$ satisfying
		$$\chi(x)\equiv 1, \ \forall\ |x|\leq1,\ \chi(x)\equiv 0, \ \forall\ |x|\geq2.$$
		Defined $\chi_n(x):=\chi(\frac{x}{R_n}),\ \eta_n(x):=\sqrt{1-\chi_n^2}$. Following the IMS formula \cite[Theorem 3.2]{Cycon} and Fatou's lemma \cite[Theorem 2.35]{Simon}, we derive that
		\begin{equation}\label{eq-trace1}
			\begin{aligned}
				\mathrm{Tr}(-\Delta\gamma_n)&=\mathrm{Tr}(-\Delta\chi_n\gamma_n\chi_n)
				+\mathrm{Tr}(-\Delta\eta_n\gamma_n\eta_n)
				+o_n(1)\\
				&\geq \mathrm{Tr}(-\Delta \gamma)+\mathrm{Tr}(-\Delta\eta_n\gamma_n\eta_n)+o(1),
			\end{aligned}
		\end{equation}
		and
		\begin{equation*}
			\begin{aligned}
				\int_{\R^3}(\rho_n*|x|^{-\alpha})\rho_n dx&=\int_{\R^3}(\chi_n^2\rho_n*|x|^{-\alpha})\chi_n^2\rho_n dx+\int_{\R^3}(\eta_n^2\rho_n*|x|^{-\alpha})\eta_n^2\rho_n dx\\
				&+2\int_{\R^3}(\chi_n^2\rho_n*|x|^{-\alpha})\eta_n^2\rho_n dx.
			\end{aligned}		
		\end{equation*}
		Defined $\tilde{\chi}_n(x)=\tilde{\chi}(\frac{x}{3R_n}),\ \tilde{\eta}_n(x)=\sqrt{1-\tilde{\chi}_n^2}$, then
		$\eta_n^2(x)=(\eta_n^2(x)-\tilde{\eta}_n^2(x))+\tilde{\eta}_n^2(x)$. We have
		\begin{equation*}
			\chi_n^2(x)|x-y|^{-\alpha}\tilde{\eta}_n^2(y)\leq\dfrac{\chi_n^2\tilde{\eta}_n^2}{||x|-|y||^{\alpha}}\leq\dfrac{1}{R_n^\alpha},
		\end{equation*}
		and
		\begin{equation*}
			\chi_n^2(x)|x-y|^{-\alpha}(\eta_n^2(y)-\tilde{\eta}_n^2(y))\leq I_{R_n<|y|<6R_n}|x-y|^{-\alpha}.
		\end{equation*}
		 Combining \eqref{A3} with the two estimates above, we obtain that
		$$\int_{\R^3}(\chi_n^2\rho_n*|x|^{-\alpha})\eta_n^2\rho_n dx\leq \frac{\|\rho_n\|_1^2}{R_n^\alpha}+\int_{\R^3}\int_{R_n\leq|x|\leq 6R_n}\dfrac{\rho_n(x)\rho_n(y)}{|x-y|^\alpha}dxdy=o_n(1).$$
		which means
		\begin{equation}\label{rnx*}
			\int_{\R^3}(\rho_n*|x|^{-\alpha})\rho_n dx=\int_{\R^3}(\chi_n^2\rho_n*|x|^{-\alpha})\chi_n^2\rho_ndx+\int_{\R^3}(\eta_n^2\rho_n*|x|^{-\alpha})\eta_n^2\rho_ndx+o_n(1).		
		\end{equation}
		Next, we will prove the two cases: (a) $0<\alpha<2$ and  $1\le q\le \frac{4-\alpha}{2-\alpha}$; (b) $2\leq\alpha<3$ and  $ q\in [1,\infty]$.
		
		For Case (a), 
		from $\|\gamma_n\|_{\mathcal{S}^q}=\Tr(-\Delta\gamma_n)=1$, we get
		\begin{equation}\label{2.10}
			1=\|\gamma_n\|_{\mathcal{S}^q}^{\frac{4-\alpha}{2}}(\Tr(-\Delta\gamma_n))^{\frac{\alpha}{2}}=\big(\Tr(\gamma_n^q)\big)^{\frac{4-\alpha}{2q}}\big(\Tr(-\Delta\gamma_n)\big)^{\frac{\alpha}{2}}.
		\end{equation}
		Recall that if $1<q<\infty$, then,
		\begin{equation}\label{2.11}
			\begin{aligned}
				\Tr(\gamma_n^q)&=\Tr\Big((\chi_n^2+\eta_n^2)\gamma_n^q\Big)\ge \Tr(\chi_n^{2q}\gamma_n^q) +\Tr(\eta_n^{2q}\gamma_n^q) = \Tr(\chi_n^q\gamma_n^q\chi_n^q)+\Tr(\eta_n^q\gamma_n^q\eta_n^q)\\
				&\ge \Tr\big((\chi_n\gamma_n\chi_n)^q\big)+\Tr\big((\eta_n\gamma_n\eta_n)^q\big)\ge \Tr(\gamma^q)+\Tr\big((\eta_n\gamma_n\eta_n)^q\big)+o_n(1).
			\end{aligned}
		\end{equation}
		Moreover, since  for $\alpha+\theta\ge 1$, we have that
		\begin{equation*}
			a^\alpha c^\theta+b^\alpha d^\theta\le (a+b)^\alpha(c+d)^\theta \text{ for  any }a,b,c,d\geq0.
		\end{equation*}
		Since $0<\alpha<2$ and  $1\le q\le \frac{4-\alpha}{2-\alpha}$, we deduce from the definition of $\mathcal{K}_{\alpha,q}^{(N)}$, \eqref{eq-trace1}, \eqref{rnx*} and  \eqref{2.11} that 
		\begin{equation*}
			\begin{aligned}
				1&\ge \big[\Tr(\gamma^q)\big]^{\frac{4-\alpha}{2q}}\big[\Tr(-\Delta\gamma)\big]^{\frac{\alpha}{2}}+\big[\Tr\big((\eta_n\gamma_n\eta_n)^q\big)\big]^{\frac{4-\alpha}{2q}}\big[\Tr(-\Delta\eta_n\gamma_n\eta_n)\big]^{\frac{\alpha}{2}}+o_n(1)\\
				&\ge [\Tr(\gamma^q)]^{\frac{4-\alpha}{2q}}[\Tr(-\Delta\gamma)]^{\frac{\alpha}{2}}+(\mathcal{K}_{\alpha,q}^{(N)})^{\frac{\alpha}{2}}\int_{\mathbb{R}^3}\big((\eta_n^2\rho_n)*{|x|^{-\alpha}}\big)\eta_n^2\rho_n dx+o_n(1)\\
				&\ge [\Tr(\gamma^q)]^{\frac{4-\alpha}{2q}}[\Tr(-\Delta\gamma)]^{\frac{\alpha}{2}}+1-(\mathcal{K}_{\alpha,q}^{(N)})^{\frac{\alpha}{2}}\int_{\mathbb{R}^3}\big((\chi_n^2\rho_n)*{|x|^{-\alpha}}\big)\chi_n^2\rho_n dx+o_n(1).
			\end{aligned}
		\end{equation*}
		Rearranging the inequality and letting $n\to \infty$, we get
		$$\mathcal{K}^{(N)}_{\alpha,q}\geq\f{\|\ga\|_{\mathcal{S}^q}^{\frac{4-\alpha}{\alpha}}\mathrm{Tr}(-\Delta\ga)}{(\inte(\rho_\ga*|x|^{-\alpha})\rho_\ga dx)^{\frac{2}{\alpha}}}.$$
		
		For Case (b), We infer from $\|\gamma_n\|_{\SQ}=1$ that $\|\gamma\|_{\SQ}\le 1$ and $\|\eta_n\gamma_n\eta_n\|_{\mathcal{S}^q}\le1$. Therefore
		\begin{equation*}
			\begin{aligned}
				1=\Big(\text{Tr}(-\Delta\gamma_n)\Big)^{\frac{\alpha}{2}} &\ge \Big(\text{Tr}(-\Delta\gamma)\Big)^{\frac{\alpha}{2}}+\Big(\text{Tr}(-\Delta\eta_n\gamma_n\eta_n)\Big)^{\frac{\alpha}{2}}+o_n(1)\\
				&\ge \|\gamma\|_{\mathcal{S}^q}^{\frac{4-\alpha}{2}}\Big(\text{Tr}(-\Delta\gamma)\Big)^{\frac{\alpha}{2}}+\|\eta_n\gamma_n\eta_n\|_{\mathcal{S}^q}^{\frac{4-\alpha}{2}}\Big(\text{Tr}(-\Delta\eta_n\gamma_n\eta_n)\Big)^{\frac{\alpha}{2}}+o_n(1)\\
				&\ge \|\gamma\|_{\mathcal{S}^q}^{\frac{4-\alpha}{2}}\Big(\text{Tr}(-\Delta\gamma)\Big)^{\frac{\alpha}{2}}+(\mathcal{K}_{\alpha,q}^{(N)})^{\frac{\alpha}{2}}\int_{\mathbb{R}^3}\Big((\eta_n^2\rho_n)*{|x|^{-\alpha}}\Big)\eta_n^2\rho_n dx+o_n(1)\\
				&\ge \|\gamma\|_{\mathcal{S}^q}^{\frac{4-\alpha}{2}}\Big(\text{Tr}(-\Delta\gamma)\Big)^{\frac{\alpha}{2}}+1-(\mathcal{K}_{\alpha,q}^{(N)})^{\frac{\alpha}{2}}\int_{\mathbb{R}^3}\Big((\chi_n^2\rho_n)*{|x|^{-\alpha}}\Big)\chi_n^2\rho_n dx+o_n(1).
			\end{aligned}
		\end{equation*}
		Combining with Case (a), 
		this implies that $\gamma\neq 0$ is an optimizer for $\mathcal{K}_{\alpha,q}^{(N)}$.
		
		\textbf{(i).} For $1\le q< \frac{4-\alpha}{(2-\alpha)_+}$, we prove that the minimizing sequence satisfying \eqref{norma1} is sequentially compact. Specifically, let $1\leq M:=\Rank \gamma\leq N$,  we intend to show that   $\|\gamma_n-\gamma\|_{\mathcal{S}^q}\overset{n}{\to}0$,
		furthermore, by rearranging the order of $\{i\}$ if necessary, there exists a subsequence such that 
		\begin{equation}\label{2.12}
			a_i:=\lim_{n\to\infty}a_{in}>0 \text{ and } u_{in}\overset{n}{\to}u_i \text{ in } H^1(\mathbb{R}^3) \text{ for  }1\le i\le M;\  a_i=0 \text{ for }M+1\le i\le N,
		\end{equation}
		and 
		\begin{equation}\label{2.13}
			\gamma=\sum_{i=1}^M a_i |u_i\rangle \langle u_i|,\ \ \text{where}\ (u_i,u_j)_{L^2}=\delta
			_{ij} \ \text{ for }\ i,j=1,2,\cdots, M.
		\end{equation}
		By \eqref{eq-trace1}, \eqref{2.10}, \eqref{2.11}, and H\"older inequality,
		\begin{equation}\label{2.14}
			\begin{aligned}
				1\ge& \Big(\Tr(\gamma^q)+\Tr((\eta_n\gamma_n\eta_n)^q)\Big)^{\frac{4-\alpha}{2q}}\Big(\Tr(-\Delta\gamma)+\Tr(-\Delta\eta_n\gamma_n\eta_n)\Big)^{\frac{\alpha}{2}}+o_n(1)\\
				=&\Big[\Big(\Tr(\gamma^q)+\Tr((\eta_n\gamma_n\eta_n)^q)\Big)^{\frac{4-\alpha}{4-\alpha+q\alpha}}\Big(\Tr(-\Delta\gamma)+\Tr(-\Delta\eta_n\gamma_n\eta_n)\Big)^{\frac{q\alpha}{4-\alpha+q\alpha}}\Big]^{\frac{4-\alpha+q\alpha}{2q}}+o_n(1)\\
				\geq&\Big\{\big[\Tr(\gamma^q)\big]^{\frac{4-\alpha}{4-\alpha+q\alpha}}\big[\Tr(-\Delta\gamma)\big]^{\frac{q\alpha}{4-\alpha+q\alpha}}
                \\
		&\quad+\big[\Tr((\eta_n\gamma_n\eta_n)^q)\big]^{\frac{4-\alpha}{4-\alpha+q\alpha}}\big[\Tr(-\Delta\eta_n\gamma_n\eta_n)\big]^{\frac{q\alpha}{4-\alpha+q\alpha}}\Big\}^{\frac{4-\alpha+q\alpha}{2q}}+o_n(1).
			\end{aligned}
		\end{equation}
		Next, we claim that 
		\begin{equation*}
			\lim_{n\to\infty}\|\eta_n\gamma_n\eta_n\|_{\mathcal{S}^q}=0.
		\end{equation*}
		Otherwise 
		\begin{equation}\label{2.16}
			\liminf_{n\to\infty}\|\eta_n\gamma_n\eta_n\|_{\mathcal{S}^q}>C_1>0.
		\end{equation}
		If $\Tr(-\Delta\eta_n\gamma_n\eta_n)\overset{n}\to0$, we then deduce from \eqref{eq:GNS} that 
		$$\int_{\mathbb{R}^3}\int_{\mathbb{R}^3}\frac{\eta_n^2\rho_n(x)\eta_n^2\rho_n(y)}{|x-y|^{\alpha}}dxdy\overset{n}\to0.$$
		It then follows from \eqref{rnx*} that 
		\begin{equation*}
			\lim_{n\to\infty}\int_{\mathbb{R}^3}\big(\rho_n*|x|^{-\alpha}\big)\rho_ndx=\int_{\mathbb{R}^3}\big(\rho_\gamma*|x|^{-\alpha}\big)\rho_\gamma dx=\big(\mathcal{K}_{\alpha,q}^{(N)}\big)^{-\frac{\alpha}{2}}.
		\end{equation*}
		Taking \eqref{2.16} into the first inequality in \eqref{2.14}, and using the GNS inequality \eqref{eq:GNS}, we get that, for some $C_2>0$,
		\begin{equation*}
			\begin{aligned}
				1&\geq \Big(\Tr(\gamma^q)+C_1^q\Big)^{\frac{4-\alpha}{2q}}\big[\Tr(-\Delta\gamma)\big]^{\frac{\alpha}{2}}\ge \big[\Tr(\gamma^q)\big]^{\frac{4-\alpha}{2q}}\big[\Tr(-\Delta\gamma)\big]^{\frac{\alpha}{2}}+C_2\\
				&\geq (\mathcal{K}_{\alpha,q}^{(N)})^{\frac{\alpha}{2}}\int_{\mathbb{R}^3}\big(\rho_\gamma*|x|^{-\alpha}\big)\rho_\gamma dx+C_2>1.
			\end{aligned} 
		\end{equation*}
		This leads to a contradiction. 
		If $\Tr(-\Delta\eta_n\gamma_n\eta_n)>C_3>0$ for some $C_3>0$.  Since  $\frac{4-\alpha+q\alpha}{2q}>1$, we  deduce from \eqref{2.14} and \eqref{2.16} that 
		\begin{equation*}
			\begin{aligned}
				1\ge& [\Tr(\gamma^q)]^{\frac{4-\alpha}{2q}}[\Tr(-\Delta\gamma]^{\frac{\alpha}{2}}+[\Tr((\eta_n\gamma_n\eta_n)^q)]^{\frac{4-\alpha}{2q}}[\Tr(-\Delta\eta_n\gamma_n\eta_n)]^{\frac{\alpha}{2}}+o_n(1)\\
				&+\frac{4-\alpha+q\alpha}{2q}[\Tr(\gamma^q)]^{\frac{4-\alpha}{2q}-\frac{4-\alpha}{4-\alpha+q\alpha}}[\Tr(-\Delta\gamma)]^{\frac{\alpha}{2}-\frac{q\alpha}{4-\alpha+q\alpha}}[\Tr((\eta_n\gamma_n\eta_n)^q)]^{\frac{4-\alpha}{4-\alpha+q\alpha}}[\Tr(-\Delta\eta_n\gamma_n\eta_n)]^{\frac{q\alpha}{4-\alpha+q\alpha}}\\
				\ge& (\mathcal{K}_{\alpha,q}^{(N)})^{\frac{\alpha}{2}}\int_{\mathbb{R}^3}\big(\rho_\gamma*|x|^{-\alpha}\big)\rho_\gamma dx+(\mathcal{K}_{\alpha,q}^{(N)})^{\frac{\alpha}{2}}\int_{\mathbb{R}^3}\big[(\eta_n^2\rho_n)*{|x|^{-\alpha}}\big]\eta_n^2\rho_n dx+o_n(1)\\
				&+\frac{4-\alpha+q\alpha}{2q}[\Tr(\gamma^q)]^{\frac{4-\alpha}{2q}-\frac{4-\alpha}{4-\alpha+q\alpha}}[\Tr(-\Delta\gamma)]^{\frac{\alpha}{2}-\frac{q\alpha}{4-\alpha+q\alpha}}[\Tr((\eta_n\gamma_n\eta_n)^q)]^{\frac{4-\alpha}{4-\alpha+q\alpha}}[\Tr(-\Delta\eta_n\gamma_n\eta_n)]^{\frac{q\alpha}{4-\alpha+q\alpha}}\\
				>&1,
			\end{aligned}
		\end{equation*}
		where we use \eqref{rnx*} in the last inequality. This indicates that \eqref{2.16} cannot hold, thus the claim is obtained. 
		This implies from GNS inequality
		that
		\begin{equation*}
			\int_{\mathbb{R}^3}\big[(\eta_n^2\rho_n)*{|x|^{-\alpha}}\big]\eta_n^2\rho_n dx=o_n(1)
		\end{equation*}
		and
		\begin{equation*}
			\int_{\mathbb{R}^3}\big(\rho_n*{|x|^{-\alpha}}\big)\rho_n dx=\int_{\mathbb{R}^3}\big(\rho_\gamma*|x|^{-\alpha}\big)\rho_\gamma dx+o_n(1).
		\end{equation*}
		As a consequence, we have  
		\begin{equation*}
			\begin{aligned}
				1=&\|\gamma_n\|_{\mathcal{S}^q}^{\frac{4-\alpha}{2}}(\Tr(-\Delta\gamma_n))^{\frac{\alpha}{2}}\ge \|\gamma\|_{\mathcal{S}^q}^{\frac{4-\alpha}{2}}(\Tr(-\Delta\gamma))^{\frac{\alpha}{2}}+o_n(1)\\
				\ge& (\mathcal{K}_{\alpha,q}^{(N)})^{\frac{\alpha}{2}}\int_{\mathbb{R}^3}\big(\rho_\gamma*|x|^{-\alpha}\big)\rho_\gamma dx+o_n(1)=1+o_n(1)
			\end{aligned}
		\end{equation*}
		and
		\begin{equation*}
			1=\Tr(\gamma_n^q)=\Tr(\gamma^q)+o_n(1)\ \text{and}\ 1=\Tr(-\Delta\gamma_n)=\Tr(-\Delta\gamma)+o_n(1).
		\end{equation*}
		It follows from Theorem 2.16 in \cite{Simon} that $\|\gamma_n-\gamma\|_{\mathcal{S}^q}\overset{n}\to 0$. In addition, from the definition of $\{\gamma_{n}\}$ and \eqref{norma1}, There exists a subsequence such that 
		\begin{equation*}
			a_i:=\lim_{n\to\infty}a_{in}\geq0 \text{ and } u_{in}\overset{n}{\rightharpoonup }u_i \text{ in } H^1(\mathbb{R}^3).
		\end{equation*}
		Thus, $\gamma^q$ can be expressed as 
		\begin{equation}\label{2.19}
			\gamma^q=\sum_{i=1}^N a_i^q |u_i\rangle \langle u_i|.
		\end{equation}
		Combining the above analysis, we have
		\begin{equation*}
			\begin{aligned}
				1=\Tr(\gamma_n^q)=\sum_{i=1}^N a_{in}^q\stackrel{n}\to\sum_{i=1}^N a_i^q=\Tr(\gamma^q)
				\leq \sum_{i=1}^N a_i^q \langle u_i|u_i\rangle \le 1.
			\end{aligned}
		\end{equation*}
		This implies that
		\begin{equation*}
			(u_i,u_i)_{L^2}=1 \text{ provided }a_i>0, \text{ for }\ i=1,2,\cdots,N. 
		\end{equation*}
		Recall $\Rank \gamma=\Rank \gamma^q=M$, rearranging the order of $\{i\}$ if necessary, we   obtain  that
		\begin{equation*}
			a_i>0 \ \text{ and }\ u_i\not=0 \ \text{ for }\ i=1,2,\cdots,M.
		\end{equation*}
		This combine with \eqref{2.19} gives that  $(u_i,u_i)_{L^2}=1$ for $i=1,2,\cdots,M$. This further implies 
		\begin{equation}\label{2.22}
			u_{in}\overset{n}\to u_i \text{ in } L^2(\mathbb{R}^3;\mathbb{C}),\ \text{for all }i=1,2,\cdots,M.
		\end{equation}
		Recall that $(u_{in},u_{jn})_{L^2}=\delta_{ij}$ for all $i,j=1,2,\cdots,N,$  and $n\in\mathbb{N}^+$, we thus deduce that
		\begin{equation*}
			(u_i,u_j)_{L^2}=\delta
			_{ij} \ \text{ for }\ i,j=1,2,\cdots, M.
		\end{equation*}
		We claim that if $M<N$, then 
		\begin{equation*}
			a_i=0 \ \text{for $i=M+1,\cdots, N$}.
		\end{equation*}
		For otherwise, if $a_{M+1}>0$, then $u_{M+1}\not=0$ .  This together with \eqref{2.22} indicates that 
		$(u_i,u_{M+1})_{L^2}=0$ for $i=1,2,\cdots, M$. It then follows from \eqref{2.19} that $\Rank \gamma\geq M+1$, which however contradicts that $\Rank \gamma^q=M$. This proves the claim. Together with \eqref{2.22} and the orthogonality established
		above, this yields \eqref{2.13}.  Moreover, recalling the above convergence of the Hartree term, we obtain
		\begin{equation*}
			\int_{\mathbb{R}^3}\Big(\rho_n*{|x|^{-\alpha}}\Big)\rho_n dx\overset{n}\to \int_{\mathbb{R}^3}\Big(\rho_\gamma*{|x|^{-\alpha}}\Big)\rho_\gamma dx={\big(\mathcal{K}_{\alpha,q}^{(N)}\big)^{-\frac{\alpha}{2}}}.
		\end{equation*}
		Therefore, 
		\begin{equation*}
			1=[\Tr (-\Delta\gamma_n)]^{\frac{\alpha}{2}}\geq[\Tr (-\Delta\gamma)]^{\frac{\alpha}{2}} \ge \big(\mathcal{K}_{\alpha,q}^{(N)}\big)^{\frac{\alpha}{2}} \frac{\int_{\mathbb{R}^3}(\rho_\gamma*|x|^{-\alpha})\rho_\gamma dx}{[\Tr(\gamma^q)]^{4-\alpha} }= 1.
		\end{equation*}
		This together with \eqref{2.22} and  the fact that
		$a_i=0$ for $i>M$ indicates that $ u_{in} \overset{n}\to  u_i $ in $H^1(\mathbb{R}^3)$  for $1\leq i\leq M$.  Hence \eqref{2.12} follows.
		
		Subsequently, we will reveal the explicit expression of the optimizer $\gamma$. 
		We start with a minimiser normalised so that
		\begin{equation}\label{nor2}
			\mathrm{Tr}(-\Delta\gamma)=\int_{\mathbb{R}^3}(\rho_\gamma*|x|^{-\alpha})\rho_\gamma dx=1,\ \mbox{and}\ \|\gamma\|_{\mathcal{S}^q}=(\mathcal{K}_{\alpha,q}^{(N)})^\frac{\alpha}{4-\alpha}.
		\end{equation} 
		We will prove this in two parts: case (A):  $0<\alpha<3$ and $1\le q<\frac{4-\alpha}{(2-\alpha)_+}(<\infty)$,  and case (B): $2\le \alpha<3$ and $q=\infty$.
		{\em Case (A) }.	Given a smooth curve of operator
		\begin{equation*}
			0\leq \gamma^*(t)=\gamma(t)=\gamma+t\delta+o(t), \quad {\rm rank~}\gamma(t)\le N
		\end{equation*}
		Substituting $\gamma(t)$ into \eqref{eq:GNS} yields
		\begin{equation}\label{2.27}
			\begin{aligned}
				(\mathcal{K}_{\alpha,q}^{(N)})^{\frac{\alpha}{2}}&\le \frac{\big[\text{Tr}(\gamma^q(t))\big]^{\frac{4-\alpha}{2q}}[\Tr(-\Delta\gamma(t))]^{\frac{\alpha}{2}}}{\int_{\mathbb{R}^3}\big(\rho_{\gamma(t)}*|x|^{-\alpha}\big)\rho_{\gamma(t)}dx}\\
				&\le \frac{\big[\Tr(\gamma^q)+qt\text{Tr}(\delta\gamma^{q-1})+o(t)\big]^{\frac{4-\alpha}{2q}}[\text{Tr}(-\Delta\gamma)+t\text{Tr}(-\Delta\delta)+o(t)]^{\frac{\alpha}{2}}}{\int_{\mathbb{R}^3}\int_{\mathbb{R}^3}[\rho_\gamma(x)+t\rho_\delta(x)+o(t)(x)][\rho_\gamma(y)+t\rho_\delta(y)+o(t)(y)]|x-y|^{-\alpha}dxdy}\\
				&= [\Tr(\gamma^q)]^{\frac{4-\alpha}{2q}}\frac{[1+qt\frac{\text{Tr}(\delta\gamma^{q-1})}{\Tr(\gamma^q)}+o(t)]^{\frac{4-\alpha}{2q}}[1+t\Tr(-\Delta\delta)+o(t)]^{\frac{\alpha}{2}}}{1+2t\int_{\mathbb{R}^3}\int_{\mathbb{R}^3}\rho_{\gamma}(x)\rho_{\delta}(y)|x-y|^{-\alpha}dxdy+o(t)}\\
				&\le (\mathcal{K}_{\alpha,q}^{(N)})^{\frac{\alpha}{2}}\Big[1+t\Tr\Big(\delta\alpha\big[-\Delta-\frac{4}{\alpha}\rho_\gamma*|x|^{-\alpha}+\Big(\frac{4-\alpha}{\alpha}\Big)\frac{\gamma^{q-1}}{\text{Tr}(\gamma^q)}\big]\Big)+o(t)\Big],
			\end{aligned}
		\end{equation}
		As seen from the above equation,
		\begin{equation}\label{2.28}
			0\le t\Tr\Big(\delta\alpha\big[-\Delta-\frac{4}{\alpha}\rho_\gamma*|x|^{-\alpha}+\Big(\frac{4-\alpha}{\alpha}\Big)\frac{\gamma^{q-1}}{\text{Tr}(\gamma^q)}\big]\Big)+o(t),\quad \text{for $|t|$ small\ enough}.
		\end{equation}
		We deduce by taking $\gamma(t):=e^{{\rm i}tH}\gamma e^{-{\rm i}tH}=\gamma+{ \text{i}}t[H,\gamma]+o(t)$, for some self-adjoint (smooth and finite rank) operator $H$,
		\begin{equation*}
			\begin{aligned}
				0&=\text{Tr}\Big([H,\gamma]\Big(-\Delta-\frac{4}{\alpha}\rho_\gamma*|x|^{-\alpha}+\Big(\frac{4-\alpha}{\alpha}\Big)\frac{\gamma^{q-1}}{\text{Tr}(\gamma^q)}\Big)\Big)\\
				&=\text{Tr}\Big(H\Big[\gamma,-\Delta-\frac{4}{\alpha}\rho_\gamma*|x|^{-\alpha}+\Big(\frac{4-\alpha}{\alpha}\Big)\frac{\gamma^{q-1}}{\text{Tr}(\gamma^q)}\Big]\Big),
			\end{aligned}
		\end{equation*}
		where $[\cdot,\cdot]$ is Lie bracket and $[a,b]=ab-ba$. From the arbitrariness of $H$, it follows that
		$$\Big[\gamma,-\Delta-\frac{4}{\alpha}\rho_\gamma*|x|^{-\alpha}+\Big(\frac{4-\alpha}{\alpha}\Big)\frac{\gamma^{q-1}}{\text{Tr}(\gamma^q)}\Big]=0.$$
		Hence $\gamma$ commutes with operator $H_{\alpha,\gamma}:=-\Delta-\frac{4}{\alpha} \rho_\gamma*|x|^{-\alpha}$. We can therefore write the explicit expression of $\gamma=\sum_{i=1}^M k_i|u_{n_i}\rangle \langle u_{n_i}|, \ \text{for some}\ u_{n_i}\ \text{ satisfying } \ H_{\alpha,\gamma} u_{n_i}=\mu_{n_i}u_{n_i},$ and some $k_i>0$. In particular $H_{\alpha,\gamma}$ admits at least $M$ eigenvalues, $1\le M\le N$.
		
		Using now $\gamma (t) = \gamma + t\delta$ for a $\delta$ supported on the range of $\gamma$ and for $t$ small enough in \eqref{2.28}, thus
		\begin{equation*}
			-\Delta u_{n_i}-\frac{4}{\alpha}\rho_\gamma*|x|^{-\alpha}u_{n_i}+\frac{4-\alpha}{\alpha}(\text{Tr}(\gamma^q))^{-1}\gamma^{q-1}u_{n_i}\equiv 0.
		\end{equation*}
		We infer that
		\begin{equation*}
			\mu_{n_i}=-\frac{4-\alpha}{\alpha}\frac{k_i^{q-1}}{\Tr(\gamma^q)}<0 \text{\  and thus \ }k_i=\big[\frac{\alpha}{4-\alpha}\Tr\big(\gamma^q)\big]^{\frac{1}{q-1}}|\mu_{n_i}|^{\frac{1}{q-1}}.
		\end{equation*}
		By the identity $\Tr(\gamma^q)=\sum\limits_{i=1}^Mk_i^q$, we have 
		$$\big(\Tr(\gamma^q)\big)^\frac{1}{q-1}=\big(\frac{4-\alpha}{\alpha}\big)^\frac{q}{q-1}\Big(\sum\limits_{k=1}^M|\mu_{n_k}|^{\frac{q}{q-1}}\Big)^{-1}.$$ Eventually,  
		for $0<\alpha<2$ and $1< q<q_{\alpha}^*$, we get that
		\begin{equation}\label{i3A}
			\gamma=\sum_{i=1}^M k_i|u_{n_i}\rangle \langle u_{n_i}|,\quad \text{ where }\quad
			k_i=\frac{4-\alpha}{\alpha}\Big(\sum_{k=1}^M|\mu_{n_k}|^{\frac{q}{q-1}}\Big)^{-1}|\mu_{n_i}|^{\frac{1}{q-1}}.
		\end{equation}
		{\em Case (B) }.  
		As in Case (A), consider the same admissible variation
		$\gamma(t)=\gamma+t\delta+o(t)$, where $\delta$ is chosen  such that $\|\gamma(t)\|\le \|\gamma\|$.  For $|t|$ small enough,
		\begin{equation*}
			\begin{aligned}
				(\mathcal{K}_{\alpha,\infty}^{(N)})^{\frac{\alpha}{2}}&\le \frac{\|\gamma(t)\|^{\frac{4-\alpha}{2}}[\Tr(-\Delta\gamma(t))]^{\frac{\alpha}{2}}}{\int_{\mathbb{R}^3}\big(\rho_{\gamma(t)}*|x|^{-\alpha}\big)\rho_{\gamma(t)}dx}
				\le (\mathcal{K}_{\alpha,\infty}^{(N)})^{\frac{\alpha}{2}}\frac{1+\alpha t\Tr(-\Delta\delta)+o(t)}{1+2t\int_{\mathbb{R}^3}(\rho_\gamma*|x|^{-\alpha})\rho_\delta dx+o(t)}\\
				&\le (\mathcal{K}_{\alpha,\infty}^{(N)})^{\frac{\alpha}{2}}\big\{1+\alpha t\Tr[\delta(-\Delta-\frac{4}{\alpha}\rho_\gamma*|x|^{-\alpha})]+o(t)\big\}.
			\end{aligned}
		\end{equation*}
		Taking $\delta={\rm i}[H,\gamma]$ for some self-adjoint and finite rank operator $H$, then we obtain that
		\begin{equation*}
			\Big[\gamma,\ -\Delta-\frac{4}{\alpha}\rho_\gamma*|x|^{-\alpha}\Big]=0.
		\end{equation*}
		Similar we have 
		\begin{equation*}
			\gamma=\sum_{i=1}^M k_i|u_{n_i}\rangle \langle u_{n_i}|, \quad \text{for some}\ u_{n_i}\ \text{ satisfying } \ H_{\alpha,\gamma} u_{n_i}=\mu_{n_i}u_{n_i}.
		\end{equation*}
		
		We claim that for all $1\leq i\le R$,
		\begin{equation}\label{2.30}
			\text{$\mu_{n_i}\leq0$ if $2<\alpha<3$, and  $\mu_{n_i}<0$ if $\alpha=2$}.
		\end{equation}
		Actually, for $0<\varepsilon <k_i$ is small enough, let
		\begin{equation*}
			\gamma'=\gamma-\varepsilon |u_{n_i}\rangle \langle u_{n_i}|.
		\end{equation*}
		Then, $\|\gamma'\|\leq \|\gamma\|=(\mathcal{K}_{\alpha,q}^{(N)})^\frac{\alpha}{4-\alpha}$, $\rho_{\gamma'}=\rho_\gamma-\varepsilon |u_{n_i}|^2$ and 
		\begin{equation*}
			\Tr(-\Delta\gamma')=\Tr(-\Delta\gamma)-\varepsilon\int_{\mathbb{R}^3}|\nabla u_{n_i}|^2dx=1-\varepsilon\int_{\mathbb{R}^3}|\nabla u_{n_i}|^2dx.
		\end{equation*}
		Moreover, we can deduce from $H_{\alpha,\gamma} u_{n_i}=\mu_{n_i}u_{n_i}$ that 
		$$\int_{\mathbb{R}^3}|\nabla u_{n_i}|^2dx=\mu_{n_i}+\frac{4}{\alpha}\int_{\mathbb{R}^3}(\rho_{\gamma}*|x|^{-\alpha})|u_{n_i}|^2dx.$$
		By the definition of $\mathcal{K}_{\alpha,\infty}^{(N)}$, we have $$(\mathcal{K}_{\alpha,\infty}^{(N)})^{\frac{\alpha}{2}}\le \frac{\|\gamma'\|^{\frac{4-\alpha}{2}}\big(\Tr(-\Delta\gamma')\big)^{\frac{\alpha}{2}}}{\int_{\mathbb{R}^3}(\rho_{\gamma'}*|x|^{-\alpha})\rho_{\gamma'}dx},$$ 
		Using \eqref{nor2}, and the above two
		estimates, we arrive at 
		\begin{equation*}
			1-2\varepsilon\int_{\mathbb{R}^3}(\rho_{\gamma}*|x|^{-\alpha})|u_{n_i}|^2dx+\varepsilon^2\int_{\mathbb{R}^3}(|u_{n_i}|^2*|x|^{-\alpha})|u_{n_i}|^2dx\le \big(1-\varepsilon \mu_{n_i}-\frac{4\varepsilon}{\alpha}\int_{\mathbb{R}^3}(\rho_{\gamma}*|x|^{-\alpha})|u_{n_i}|^2dx\big)^{\frac{\alpha}{2}}.
		\end{equation*}
		Applying Taylor's expansion of $(1-x)^\alpha$, we can obtain that 
		\begin{equation*}
			\mu_{n_i}\leq -\varepsilon\int_{\mathbb{R}^3}(|u_{n_i}|^2*|x|^{-\alpha})|u_{n_i}|^2dx \text{\ for \ } \alpha=2,
		\end{equation*}
		and 
		\begin{equation*}
			\alpha\mu_{n_i}\leq -\varepsilon\int_{\mathbb{R}^3}(|u_{n_i}|^2*|x|^{-\alpha})|u_{n_i}|^2dx+O(\varepsilon) \text{\ for \ } \alpha\in(2,3).
		\end{equation*}
		These two estimates  indicate  that claim \eqref{2.30} holds.

		Next, we point out when $2\le \alpha<3$ and $q=\infty$, then, $k_i\equiv k_j, \forall i,j\in\{1\leq l\leq M: \mu_{n_l}<0\}$. We agree by negative, there exists $k_m,\ \varepsilon,$ such that $\ k_m+\varepsilon<\|\gamma\|$. Then we construct
		$\gamma(\varepsilon)=\gamma+\varepsilon|u_i\ra\la u_i|.$
		Similar to the discussion above, by $\mu_i<0$, we get
		\begin{equation*}
			\begin{aligned}
				(\mathcal{K}_{\alpha,\infty}^{(N)})^{\frac{\alpha}{2}}&\le \frac{\|\gamma(\varepsilon)\|^{\frac{4-\alpha}{2}}(\Tr(-\Delta\gamma(\varepsilon)))^{\frac{\alpha}{2}}}{\int_{\mathbb{R}^3}(\rho_{\gamma(\varepsilon)}*|x|^{-\alpha})\rho_{\gamma(\varepsilon)}dx}
				\leq(\mathcal{K}_{\alpha,\infty}^{(N)})^{\frac{\alpha}{2}}\frac{1+\frac{\varepsilon\alpha}{2}\mu_{m}+2\varepsilon\int_{\mathbb{R}^3}(\rho_\gamma*|x|^{-\alpha})|u_m|^2dx+o(\varepsilon)}{\int_{\mathbb{R}^3}\big[(\rho_\gamma+\varepsilon|u_m|^2)*|x|^{-\alpha}\big](\rho_\gamma+\varepsilon|u_m|^2)dx}\\
				&\le (\mathcal{K}_{\alpha,\infty}^{(N)})^{\frac{\alpha}{2}}(1+\frac{\varepsilon\alpha}{2}\mu_m+o(\varepsilon))<(\mathcal{K}_{\alpha,\infty}^{(N)})^{\frac{\alpha}{2}}.
			\end{aligned}
		\end{equation*}
		Which is a contradiction, so that $k_i\equiv k_j$.
		
		Finally, we prove that $H_{\alpha, \gamma}$ has at most $M $ negative eigenvalue, where $M<N$. If not, there exists $M+1\le N$ and $\mu_{n_{M+1}}<0$, then for $\varepsilon>0$ small enough, we construct $\gamma(\varepsilon)=\gamma+\varepsilon|u_{n_{M+1}}\ra\la u_{n_{M+1}}|.$
		Similar with \eqref{2.27}, this gives
		\begin{equation*}
			\begin{aligned}
				(\mathcal{K}_{\alpha,q}^{(N)})^{\frac{\alpha}{2}} &\le (\mathcal{K}_{\alpha,q}^{(N)})^{\frac{\alpha}{2}}\Big(1+\alpha \varepsilon\Big(u_{n_{M+1}},\big(-\Delta-\frac{4}{\alpha}\rho_\gamma*|x|^{-\alpha}\big)u_{n_{M+1}}\Big)+o(\varepsilon)\Big)\\
				&\le (\mathcal{K}_{\alpha,q}^{(N)})^{\frac{\alpha}{2}}(1+\varepsilon\mu_{M+1}+o(\varepsilon))<(\mathcal{K}_{\alpha,q}^{(N)})^{\frac{\alpha}{2}}.
			\end{aligned}
		\end{equation*}
		a contradiction. Thus $H_{\alpha, \gamma}$ has at most $M$ negative eigenvalue, where $M<N$.
		
		In the previous section, we have provided the specific expressions for $k_i$ in \eqref{i3A} for the case $0<\alpha<3$ and $1< q<\frac{4-\alpha}{(2-\alpha)_+}$. For case $\alpha=2,\ q=\infty$, $k_i$ can express by \eqref{2.30} and $k_i\equiv k_j, \forall i,j\in\{1\leq l\leq M: \mu_{n_l}<0\}$.
		
		It still remains to prove that $\{\mu_{n_i}\}_{1\le i\le R}$ are the $R$ first negative eigenfunctions of $H_{\alpha,\gamma}$ under either of the following conditions $0<\alpha\le 2$, or $2<\alpha<3$ and $\Rank \gamma<N$.
		We proceed by contradiction. Suppose that the $i$-th eigenfunction of $H_{\alpha,\gamma}$ (which corresponds to the eigenvalue satisfying $\mu_i<\mu_{n_i}<0$) does not lie in the range of $\gamma$.
		
		For the first condition: $0<\alpha\le 2$,
		we define
		\begin{equation*}
			\gamma':=\gamma-k_i|u_{n_i}\rangle \langle u_{n_i}|+k_i|u_i\rangle \langle u_i|:=\gamma+\delta.
		\end{equation*}
		Then, $\|\gamma\|_{\mathcal{S}^q}=\|\gamma'\|_{\mathcal{S}^q}$, and we have  
		\begin{equation*}
			\begin{aligned}
				0<\text{Tr}(-\Delta\gamma')&=\text{Tr}(-\Delta\gamma)+k_i\langle u_i,-\Delta u_i\rangle-k_i\langle u_{n_i},-\Delta u_{n_i}\rangle\\
				&=1+\frac{4k_i}{\alpha}\int_{\mathbb{R}^3}(\rho_\gamma*|x|^{-\alpha})(|u_i|^2-|u_{n_i}|^2)+(\mu_i-\mu_{n_i})k_i\\
				&<1+\frac{4k_i}{\alpha}\int_{\mathbb{R}^3}(\rho_\gamma*|x|^{-\alpha})(|u_i|^2-|u_{n_i}|^2).
			\end{aligned}
		\end{equation*}
		We deduce from the non-negativity properties of the Hartree-type convolution \cite{LenzmannLewin} that
		\begin{equation*}
			\begin{aligned}
				\int_{\mathbb{R}^3}(\rho_{\gamma'}*|x|^{-\alpha})\rho_{\gamma'}dx=&\int_{\mathbb{R}^3}(\rho_{\gamma}*|x|^{-\alpha})\rho_{\gamma}dx+2k_i\int_{\mathbb{R}^3}(\rho_{\gamma}*|x|^{-\alpha})(|u_i|^2-|u_{n_i}|^2)dx\\
				&+k_i^2\int_{\mathbb{R}^3}\big[(|u_i|^2-|u_{n_i}|^2)*|x|^{-\alpha}\big](|u_i|^2-|u_{n_i}|^2)dx\\
				\ge&1+2k_i\int_{\mathbb{R}^3}(\rho_{\gamma}*|x|^{-\alpha})(|u_i|^2-|u_{n_i}|^2)dx>1-\f\alpha2\ge0.
			\end{aligned}
		\end{equation*}
		
		Since $\alpha\leq2$, we see that
		\begin{equation*}
			\begin{aligned}
				(\mathcal{K}_{\alpha,q}^{(N)})^{\frac{\alpha}{2}}\leq \frac{\|\gamma'\|_{\mathcal{S}^q}^{\frac{4-\alpha}{2}}(\Tr(-\Delta\gamma'))^{\frac{\alpha}{2}}}{\int_{\mathbb{R}^3}\rho_{\gamma'}*|x|^{-\alpha}\rho_{\gamma'}dx}&<(\mathcal{K}_{\alpha,q}^{(N)})^{\frac{\alpha}{2}}\frac{\Big(1+\frac{4k_i}{\alpha}\int_{\mathbb{R}^3}(\rho_\gamma*|x|^{-\alpha})(|u_i|^2-|u_{n_i}|^2)dx\Big)^{\frac{\alpha}{2}}}{1+2k_i\int_{\mathbb{R}^3}(\rho_{\gamma}*|x|^{-\alpha})(|u_i|^2-|u_{n_i}|^2)dx}<(\mathcal{K}_{\alpha,q}^{(N)})^{\frac{\alpha}{2}}.
			\end{aligned}
		\end{equation*}
		Which is a contradiction, hence $\mu_{n_i}=\mu_i$, for all $i$.
		
		For the second condition: $2<\alpha<3$ and rank $\gamma<N$,  we set
		\begin{equation*}
			\gamma':=\gamma-\varepsilon|u_{n_i}\rangle \langle u_{n_i}|+\varepsilon |u_{i}\rangle \langle u_{i}|,
		\end{equation*}
		where $0<\varepsilon<k_i$ is small enough. Then, $\Rank \gamma'\leq N$ and  $\|\gamma'\|_{\mathcal{S}^q}\leq\|\gamma\|_{\mathcal{S}^q}=(\mathcal{K}_{\alpha,q}^{(N)})^\frac{4-\alpha}{\alpha}$ for $q\geq1$. Moreover, by using $\mu_i<\mu_{n_i}<0$, a contradiction can still be derived:
		\begin{equation*}
			\begin{aligned}
				(\mathcal{K}_{\alpha,q}^{(N)})^{\frac{\alpha}{2}}
				&\leq \frac{\|\gamma'\|_{\mathcal{S}^q}^{\frac{4-\alpha}{2}}(\Tr(-\Delta\gamma'))^{\frac{\alpha}{2}}}{\int_{\mathbb{R}^3}\rho_{\gamma'}*|x|^{-\alpha}\rho_{\gamma'}dx}
				<(\mathcal{K}_{\alpha,q}^{(N)})^{\frac{\alpha}{2}}\frac{\Big(1+\frac{4\varepsilon}{\alpha}\int_{\mathbb{R}^3}(\rho_\gamma*|x|^{-\alpha})(|u_i|^2-|u_{n_i}|^2)dx+(\mu_i-\mu_{n_i})\varepsilon\Big)^{\frac{\alpha}{2}}}{1+2\varepsilon\int_{\mathbb{R}^3}(\rho_{\gamma}*|x|^{-\alpha})(|u_i|^2-|u_{n_i}|^2)dx}\\
				<&(\mathcal{K}_{\alpha,q}^{(N)})^{\alpha}\frac{1+2\varepsilon\int_{\mathbb{R}^3}(\rho_{\gamma}*|x|^{-\alpha})(|u_i|^2-|u_{n_i}|^2)dx+\frac{\alpha}{2}\varepsilon (\mu_i-\mu_{n_i})+o(\varepsilon)}{1+2\varepsilon\int_{\mathbb{R}^3}(\rho_{\gamma}*|x|^{-\alpha})(|u_i|^2-|u_{n_i}|^2)dx} 
				<(\mathcal{K}_{\alpha,q}^{(N)})^{\alpha}.
			\end{aligned}
		\end{equation*}
		
		(III) Inspired by \cite{Lewin}, we first claim that when $0<\alpha\le2$, $q$ is close to $q_\alpha^*$ sufficiently, then 
		\begin{equation}\label{iii}
			\mathcal{K}_{\alpha,q}^{(2N)}<\mathcal{K}_{\alpha,q}^{(N)} \ \text{ provided $\mathcal{K}_{\alpha,q}^{(N)}$ has an optimizer $\gamma$ of rank $N$.}
		\end{equation}
		Indeed, let $\gamma:=\sum_{j=1}^N k_j|u_j\rangle \langle u_j|$ be a minimizer  for $\mathcal{K}_{\alpha,q}^{(N)}$ of rank $N$,  such that
		\begin{equation*}
			\text{Tr}(-\Delta \gamma)=\int_{\mathbb{R}^3}(\rho_\gamma*|x|^{-\alpha})\rho_\gamma dx=1
		\end{equation*}
		and $(\mathcal{K}_{\alpha,q}^{(N)})^{\frac{\alpha}{4-\alpha}}=(\Tr (\gamma^q))^{\frac{1}{q}}$, if $q\not =\infty$, $\mathcal{K}_{2,\infty}^{(N)}=\|\gamma\|$, if $q=\infty$.
		Furthermore, there exist $N$ functions $u_1,\ u_2,...,u_N$ satisfy
		\begin{equation*}
			-\Delta u_j-\frac{4}{\alpha}(\rho_\gamma*|x|^{-\alpha})u_j=\mu_j u_j,\quad \forall\ 1\le j\le N.
		\end{equation*}
		
		For $R>0$, we define $u_{j,R}(x):=u_j(x-Re_1)$ with $e_1=(1,0,...,0)$, and the Gram matrix
		\begin{equation*}
			G_{R}= \left( \begin{array}{cc}
				E_{N\times N} & E^{R} \\
				(E^{R})^* & E_{N\times N}
			\end{array}\right)
		\end{equation*}
		with $E_{N\times N}$ is $N\times N$ identity matrix and $E^{R}=(E_{ij}^{R})_{N\times N}$ where
		\begin{equation*}
			E_{ij}^{R}:=( u_i, u_{j,R})_{L^2}=\int_{\mathbb{R}^3}u_i(x)u_j(x-Re_1)dx.
		\end{equation*}
		Next, we estimate $|E_{ij}^{R}|$. By Proposition \ref{decay}, we get
		\begin{equation*}
			\begin{aligned}
				|E_{ij}^{R}|&\le \int_{\mathbb{R}^3}|u_i(x)||u_{j}(x-Re_1)|dx\le  \int_{\mathbb{R}^3} C_i |x|^{-2}e^{-\sqrt{-\mu_i}\abs{x}}e^{-\sqrt{-\mu_j}\abs{x-Re_1}}dx=O({R^{-\infty}}).
			\end{aligned}
		\end{equation*}
		We  construct
		\begin{equation*}
			\left(\begin{array}{cc}
				\phi_{1,R}\\
				\vdots\\
				\phi_{N,R}\\
				\phi_{N+1, R}\\
				\vdots\\
				\phi_{2N,R}
			\end{array}\right)=(G_{R})^{-\frac{1}{2}}
			\left(\begin{array}{cc}
				u_1\\
				\vdots\\
				u_N\\
				u_{1,R}\\
				\vdots\\
				u_{N,R}
			\end{array}\right)
		\end{equation*}
		and
		\begin{equation*}
			\gamma_{R}=\sum_{i=1}^Nk_i\Big(|\phi_{i,R}\rangle\langle\phi_{i,R}|+|\phi_{N+i,R}\rangle \langle \phi_{N+i,R}|\Big).
		\end{equation*}
		We have
		\begin{equation*}
			\text{Tr}(\gamma_{R}^q)=2\text{Tr}(\gamma^q).
		\end{equation*}
		Substitute $\gamma_{R}$ into the inequality (\ref{eq:GNS}), which yields
		\begin{equation*}
			(\mathcal{K}_{\alpha,q}^{(2N)})^{\frac{\alpha}{2}}\le (\mathcal{K}_{\alpha,q}^{(N)})^{\frac{\alpha}{2}}\frac{2^{\frac{4-\alpha}{2q}}(\text{Tr}(-\Delta\gamma_{R}))^{\frac{\alpha}{2}}}{\int_{\mathbb{R}^3}(\rho_{\gamma_{R}}*|x|^{-\alpha})\rho_{\gamma_{R}}dx}
		\end{equation*}
		and
		\begin{equation*}
			\mathcal{K}_{2,\infty}^{(2N)}\le \mathcal{K}_{2,\infty}^{(N)}\frac{\text{Tr}(-\Delta\gamma_{R})}{\int_{\mathbb{R}^3}(\rho_{\gamma_{R}}*|x|^{-2})\rho_{\gamma_{R}}dx}.
		\end{equation*}
		By Taylor expansion of the Gram matrix $G_{R}$ as in \cite{Lewin}, we obtain that
		\begin{equation*}
			(G_{R})^{-\frac{1}{2}}=\left( \begin{array}{cc}
				E_{N\times N} &0 \\
				0 & E_{N\times N}
			\end{array}\right)-\frac{1}{2}\left( \begin{array}{cc}
				0 & E^{R} \\
				(E^{R})^* & 0
			\end{array}\right)+o((\max_{i,j}E_{ij}^{R})^2),
		\end{equation*}
		where
		\begin{equation*}
			\max_{i,j}E_{ij}^{R}=\max_{i,j}\int_{\mathbb{R}^3}|u_i(x)||u_j(x-Re_1)|dx=O(R^{-\infty}).
		\end{equation*}
		Then
		\begin{equation*}
			\gamma_{R}=\gamma+\gamma_{R}'-\frac{1}{2}\sum_{i=1}^N\sum_{j=1}^N (k_i+k_j)E_{ij}^{R}(|u_{i}\rangle\langle u_{j,R}|+|u_{j,R}\rangle\langle u_{i}|)+O(R^{-\infty}),
		\end{equation*}
		where $\gamma_{R}'=\sum_{i=1}^N k_i|u_{i,R}\rangle \langle u_{i,R}|.$
		It can be obtained through simple calculations that
		\begin{equation}\label{2.35}
			\begin{aligned}
				\text{Tr}(-\Delta{\gamma_{R}})=&\text{Tr}(-\Delta{\gamma})+\text{Tr}(-\Delta{\gamma_{R}'})\\
                -&\sum_{i=1}^N\sum_{j=1}^N (k_i+k_j)E_{ij}^{R}\Big(\mu_iE_{ij}^{R}
				+\frac{4}{\alpha}\int_{\mathbb{R}^3}(\rho_\gamma*|x|^{-\alpha})u_iu_{j,R}dx\Big)
				+O(R^{-\infty})\\
				\le& 2+\frac{4}{\alpha}\sum_{i=1}^N\sum_{j=1}^N(k_i+k_j)E_{ij}^{R}\|\rho_\gamma*|x|^{-\alpha}\|_{\infty}\int_{\mathbb{R}^3}u_iu_{j,R}dx+O(R^{-\infty})\\
				\le& 2+O(R^{-\infty}),
			\end{aligned}
		\end{equation}
		and
		\begin{equation*}
			\rho_{\gamma_{R}}=\rho_\gamma+\rho_{\gamma_{R}'}-\sum_{i=1}^N\sum_{j=1}^N(k_i+k_j)E_{ij}^{R}u_i u_{j,R}+O(R^{-\infty}).
		\end{equation*}
		Substituting this expansion into the Hartree term, we have
		\begin{equation*}
			\begin{aligned}
				\int_{\mathbb{R}^3}(\rho_{\gamma_{R}}*|x|^{-\alpha})\rho_{\gamma_{R}}=&2\int_{\mathbb{R}^3}(\rho_\gamma*|x|^{-\alpha})\rho_\gamma dx+2\sum_{i=1}^N\sum_{j=1}^N k_i k_j\int_{\mathbb{R}^3}(u_i^2*|x|^{-\alpha})u_{j,R}^2dx\\
				&-2\sum_{i=1}^N\sum_{j=1}^N(k_i+k_j)E_{ij}^{R}\int_{\mathbb{R}^3}\big((\rho_\gamma+\rho_{\gamma_R'})*|x|^{-\alpha}\big)u_i u_{j,R}dx+O(R^{-\infty})\\
				\ge&2\int_{\mathbb{R}^3}(\rho_\gamma*|x|^{-\alpha})\rho_\gamma dx+2k_1^2\int_{\mathbb{ R}^3}(u_1^2*|x|^{-\alpha})u_{1,R}^2dx+		O({R}^{-\infty}).
			\end{aligned}
		\end{equation*}
		By Proposition \ref{decay}, it follows from
		\begin{equation*}
			\int_{\mathbb{R}^3}\int_{\mathbb{R}^3}\frac{u_1^2(x)u_1^2(y-Re_1)}{|x-y|^{\alpha}}dx 
			\ge  \frac{C}{R^\alpha}
		\end{equation*}
		that
		\begin{equation*}
			\int_{\mathbb{R}^3}(\rho_{\gamma_{R}}*|x|^{-\alpha})\rho_{\gamma_{R}}\ge2\int_{\mathbb{R}^3}(\rho_\gamma*|x|^{-\alpha})\rho_\gamma dx+\frac{C}{R^\alpha}.
		\end{equation*}
		Therefore, using
		$\Tr(\gamma_R^q)=2\Tr(\gamma^q)$,
		\eqref{2.35} and the above inequality, we obtain
		\begin{equation*}
			\begin{aligned}
				(\mathcal{K}_{\alpha,q}^{(2N)})^{\frac{\alpha}{2}}&\le (\mathcal{K}_{\alpha,q}^{(N)})^{\frac{\alpha}{2}}\frac{2^{\frac{4-\alpha}{2q}}(2+O(R^{-\infty}))^{\frac{\alpha}{2}}}{2+2k_1^2CR^{-\alpha}}=2^{\frac{4-\alpha}{2q}+\frac{\alpha}{2}-1}(\mathcal{K}_{\alpha,q}^{(N)})^{\frac{\alpha}{2}}(1-k_1^2CR^{-\alpha}+O(R^{-2\alpha}))< (\mathcal{K}_{\alpha,q}^{(N)})^{\frac{\alpha}{2}},
			\end{aligned}
		\end{equation*}
		for $0<\alpha\le2$, $q$ is close to $q_\alpha^*$ sufficiently, and $R$ is large enough. Hence, \eqref{iii} is obtained. The proof of case $q=q_\alpha^*$  is similar, and we omit the proof for this part.
		
		Finally, we will prove \eqref{2n<n}. 
		By contradictiony, if $\mathcal{K}_{\alpha,q}^{(N)}=\mathcal{K}_{\alpha,q}^{(2N)}$ for some $N\in \mathbb{N}^+$. Then, there exists a $\gamma$ with ${\rm rank~}\gamma=M\leq N$ such that $\gamma$ is a minimizer of (\ref{eq:GNS}) with $\mathcal{K}_{\alpha,q}^{(M)}=\mathcal{K}_{\alpha,q}^{(N)}=\mathcal{K}_{\alpha,q}^{(2N)}$. Thus we have $\mathcal{K}_{\alpha,q}^{(M)}=\mathcal{K}_{\alpha,q}^{(2M)}\ $, which however contradicts  \eqref{iii}. 
		
	\end{proof}
	
	
	
	\section{Existence and non-existence of minimizers}\label{sec:3}
	In this section, we study the variational problem \eqref{eq:mainvar}
	and determine the role of the critical coupling constant
	$\mathcal K^{(N)}$. The finite-rank GNS inequality obtained in
	Section~\ref{sec:2} provides lower boundedness of \eqref{eq:energy} and also controls the possible ranks of minimizers. We shall establish Theorem \ref{thm:existence} on the existence and non-existence of minimizers for $e_K(N)$ defined by \eqref{eq:mainvar}.
	
	\begin{proof}[Proof of Theorem \ref{thm:existence}]
		\textbf{1.} We first prove that energy has a lower bound, where $0<K<\mathcal{K}^{(N)}$.
		Let $\ga=\sum_{i=1}^M|u_i\ra\la u_i|$ be an operator, since $V(x)\geq 0$, we obtain from \eqref{eq:critical} that for any $0<K<\mathcal{K}^{(N)}$,
		\begin{equation}\label{2.1}
			E_k(\ga)=\mathrm{Tr}\big(-\Delta+V(x)\big)\ga-K\inte(\rho_{\gamma}*|x|^{-2})\rho_{\gamma}dx\geq\l(1-\f{K}{\mathcal{K}^{(N)}}\r)\mathrm{Tr}(-\Delta\ga)+\inte V(x)\rho_\ga(x)dx\geq 0,
		\end{equation}
		
		Let $\{\ga_n\}$ be a minimizing sequence of $e_K(N)$ satisfying  $\ga_n=\sum_{i=1}^{m_n}|u_i^n\ra\la u_i^n|$, $u_i^n\in\mathcal{H}$, $(u_i^n,u_j^n)=\delta_{ij}$ for $i,j=1,\cdots,m_n$, and $\lim_{n\to\infty}E_K(\ga_n)=e_K(N)$.  Since $1\leq m_n\leq N$ is an integer, up to a subsequence, we may assume that $m_n\equiv M \in [1,N]$. We derive from  \eqref{2.1} that $\{u_i^n\}$ is bounded uniformly in $\mathcal{H}$ for $1\leq i \leq M$. Following compactness result (see e.g. \cite[Theorem XIII.67]{Reed1}), we obtain that there exists a function $u_i(x)\in\mathcal{H}$ such that for $1\leq i \leq M$,
		\begin{equation*}
			u_i^n\rightharpoonup u_i\ \ \hbox{weakly in}\ \ \mathcal{H}\ \ \hbox{and}\ \ u_i^n\to u_i\ \ \hbox{strongly in}\ \ L^q(\R^3)\ \ \hbox{as}\ \ n\to\infty,\ \ 2\leq q<6.
		\end{equation*}
		Therefore, we have
		\begin{equation}\label{2.3}
			\rho_{\ga_n}
			=\sum\limits_{i=1}^{M}|u_i^n|^2\to\rho_\ga=\sum\limits_{i=1}^M|u_i|^2\ \ \hbox{strongly in}\ \ L^r(\R^3)\ \ \hbox{as}\ \ n\to\infty,\ \ 1\leq r<3,
		\end{equation}
		where $\ga:=\sum\limits_{i=1}^M|u_i\ra\la u_i|$, and $u_i\in\mathcal{H}$ satisfies $(u_i,u_j)=\delta_{ij}$ for $i,j=1,\cdots,M$. 
		Moreover, by the weak lower semi-continuity, we obtain from \eqref{2.3} that
		\begin{equation*}
			e_K(N)=\lim\limits_{n\to\infty} E_K(\ga_n)\geq E_k(\ga).
		\end{equation*}
		Combined with the definition of $e_K(N)$, it can be explained that $\ga$ is a minimizer of $e_K(N)$ for $0<K<\mathcal{K}^{(N)}$.
		We then conclude that there exists at least one minimizer of $e_K(N)$ for any  $0<K<\mathcal{K}^{(N)}$.

		For any $0<K<\mathcal{K}^{(N)}$, assume  $\ga$ is a minimizer of $e_K(N)$. Similar to the argument of \cite[Appendix A]{Sol}, $\ga$ can be written in the form $\ga=\sum_{i=1}^M|u_{k_i}\ra\la u_{k_i}|$, where $u_{k_i}$ is an eigenfunction of the operator $$H_V=-\Delta+V(x)-2K(\rho_\ga*|x|^{-2})\ \ \mbox{in} \ \ \R^3,$$
		and corresponds to the $k_i$-th eigenvalue $\mu_{k_i}$. This gives that $u_{k_i}$ satisfies
		\begin{equation}\label{3.3}
			-\Delta u_{k_i}+V(x)u_{k_i}-2K(\rho_\ga*|x|^{-2})u_{k_i}=\mu_{k_i}u_{k_i},\ \ 1\leq i \leq M,
		\end{equation}
		where $\rho_\ga(x)=\sum_{j=1}^M|u_{k_j}|^2$.
		In the following, we prove that $(u_{k_1},u_{k_2})$ is a ground state of \eqref{eq:EL}.
		Noting from \cite[Theorem 11.8]{Lieb} that $\mu_1< \cdots < \mu_M$, it suffices to show that $\mu_{k_i}=\mu_i$ holds for $1\leq i \leq M$, which are the i-first eigenvalues of the operator $H_V$.
		
		We first prove that $\mu_{k_1}=\mu_1$.
		If $\mu_{k_1}\neq\mu_1$, then $\mu_1<\mu_{k_1}\leq\mu_{k_2}$. Hence,  there is an eigenfunction $u_1\in\mathcal{H}$ of $H_V$ in $\R^3$, which corresponds to the first eigenvalue $\mu_1$.
		Define the operator
		\begin{equation*}
			\ga':=\ga-|u_{k_1}\ra\la u_{k_1}|+|u_1\ra\la u_1|
			=|u_1\ra\la u_1|+\sum_{i=2}^{M}|u_{k_i}\ra\la u_{k_i}|.
		\end{equation*}
		We then calculate from \eqref{3.3} that
		\begin{equation*}
			\begin{aligned}
				\mathrm{Tr}(-\Delta \ga')=&\mathrm{Tr}(-\Delta\ga)-\inte|\nabla u_{k_1}|^2dx+\inte |\nabla u_1|^2dx\\
				=&\mathrm{Tr}(-\Delta\ga)+\inte V(x)\big(|u_{k_1}|^2-|u_1|^2\big)dx
				+2K\inte(\rho_\ga*|x|^{-2})\big(|u_1|^2-|u_{k_1}|^2\big)dx+\mu_1-\mu_{k_1},
			\end{aligned}
		\end{equation*}
		and
		\begin{equation*}
			\mathrm{Tr}(V(x)\ga')=\mathrm{Tr}(V(x)\ga)+\inte V(x)\big(|u_1|^2-|u_{k_1}|^2\big)dx.
		\end{equation*}
		Moreover
		\begin{equation*}
			\inte (\rho_\ga'*|x|^{-2})\rho_\ga'dx
			\geq \inte(\rho_\ga*|x|^{-2})\rho_\ga dx+2\inte(\rho_\ga*|x|^{-2})\big(|u_1|^2-|u_{k_1}|^2\big)dx.
		\end{equation*}
		Since $\mu_1<\mu_{k_1}$, we now conclude from above that
		\begin{equation*}
			\begin{aligned}
				e_K(N)\leq E_K(\ga')
				\leq E_k(\ga)+\mu_1-\mu_{k_1}<E_k(\ga)=e_K(N),
			\end{aligned}
		\end{equation*}
		a contradiction. We hence obtain that $\mu_{k_1}=\mu_1$.
		
		Similar with the above analysis,  considering the following operator
		\begin{equation*}
			\ga'_i:=\ga-|u_{k_i}\ra\la u_{k_i}|+|u_i\ra\la u_i|,\ \ 1\leq i \leq M,
		\end{equation*}
		yields again a contradiction. This proves that $\mu_{k_i}=\mu_i$ for $1\leq i \leq M$.
		
		At last, we prove that if $\mu_{i}<0$, for $1\leq i \leq r_K < N$, then $\mu_{r_K+1}>0$. Indeed, if $\mu_{r_K+1}\leq 0$, then we can define $\ga_i:=\ga+|u_{r_k+1}\ra\la u_{r_k+1}|$. Similar with the above analysis, we can obtain that 
		$$e_K(N)\leq E_K(\ga')
		\leq E_k(\ga)+\mu_{r_K+1}-K\inte(|u_{r_k+1}|^2*|x|^{-2})|u_{r_k+1}|^2dx<e_K(N).$$

		\textbf{2.} In this part, we first prove that if $K\searrow 0$, then then $\Rank \gamma_K\equiv1$. Let $\mu_{v}>0$ be the first  eigenvalue of $-\Delta+V(x)$,  and $v(x)$ be the corresponding eigenfunctions. Define $\gamma_1:=|v\rangle\langle v|$, thus we have
		\begin{equation*}
			e_K(N)\le E(\gamma_1)= \mu_v-K\int_{\mathbb{R}^3}(v^2*|x|^{-2})v^2dx\le \mu_v+O(K) \quad \text{ as\quad } K\searrow0.
		\end{equation*}
		If there exist a minimizer $\gamma_K$ satisfies $\Rank \gamma_K>1$, by \eqref{eq:critical}, we know that there exists $C>0$ independent of $K$ such that 
		\begin{equation*}
			\int_{\mathbb{R}^3}(\rho_{\gamma_K}*|x|^{-2})\rho_{\gamma_K} dx\le C\quad \text{as\quad } K\searrow 0. 
		\end{equation*}
		Then we obtain that
		\begin{equation*}
			e_K(N)= \Tr(-\Delta\gamma_K)+\int_{\mathbb{R}^3}V(x)\rho_{\gamma_K} dx+O(K)\ge 2\mu_V+O(K) \quad \text{as\quad } K\searrow0.
		\end{equation*}
		Which is a contradiction, Thus,  $\Rank\gamma_K\equiv1$ for $K>0$ is small enough. Moreover, by \cite[Appendix A]{GuoZengZhou}, we know that the  minimizer for  (\ref{eq:mainvar})  is unique when $K>0$ is small enough. 
		
		Next, we consider the case $K\nearrow \mathcal{K}^{(N)}$. Let $\gamma_0=\sum_{i=1}^m |Q_i\rangle \langle Q_i|$ with rank $\gamma_0=m$ be a minimizer  of $\mathcal{K}^{(N)}$, where $Q_i$ satisfies
		\begin{equation*}
			-\Delta Q_i-2\mathcal{K}^{(N)}(\rho_{\gamma_0}*|x|^{-2})Q_i=\mu_iQ_i,\quad \forall 1\le i\le r.
		\end{equation*}
		Take a non-negative function $\varphi(x)\in C_0^\infty(\R^3,[0,1])$, such that $\varphi(x)\equiv 1$ for $|x|\leq1$ and $\varphi(x)\equiv0$ for $|x|\geq 2$.
		For any $x_0\in\R^3$ and $\tau>0$, define
		\begin{equation}\label{3.4}
			Q_i^\tau(x):=A_i^\tau\tau^{\f32}\varphi(x-x_0)Q_i\big(\tau(x-x_0)\big),\ \ 1\leq i \leq m,
		\end{equation}
		where $A_i^\tau>0$ is chosen such that $\inte |Q_i^\tau|^2dx=1$ for $1\leq i \leq M$. By the exponential decay of $|Q_i|$ and $(Q_i,Q_j)_{L^2}=\delta_{ij}$ for $i,j=1,\cdots,m$, we obtain that
		\begin{equation}\label{3.5}
			(A_i^\tau)^2=1+Ce^{-2\sqrt{|\mu_i|}\tau}, \ \hbox{and}\ \ E_{ij}:=(Q_i^\tau,Q_j^\tau)_{L^2}=O(\tau^{-\infty})\ \ \hbox{as}\ \ \tau\to\infty,\ \mbox{for}\ \ 1\leq i, j \leq N.
		\end{equation}
		It follows from \eqref{3.5} that the Gram matrix
		\begin{equation*}
			G_{\tau}:=\left( \begin{matrix}
				1 &(Q_1^{\tau}, Q_2^{\tau})_{L^2}& \cdots & (Q_1^{\tau}, Q_m^{\tau})_{L^2} \\
				(Q_2^{\tau},Q_1^{\tau})_{L^2} & 1& \cdots & (Q_2^{\tau}, Q_m^{\tau})_{L^2} \\
				\vdots & \vdots& \ddots  & \vdots\\
				(Q_m^{\tau}, Q_1^{\tau})_{L^2} & (Q_m^{\tau}, Q_2^{\tau})_{L^2} & \cdots & 1
			\end{matrix}\right)_{m\times m}
		\end{equation*}
		is positive definite for $\tau>0$ large enough.
		Hence, for such $\tau$, defining
		\begin{equation}\label{3.7}
			\left(
			\widetilde Q_1^\tau, \cdots,  \widetilde Q_m^\tau 
			\right)
			:=\left(
			Q_1^\tau ,\cdots, Q_2^\tau 
			\right)
			G_\tau^{-\f12},\quad \widetilde\ga_\tau:=\sum\limits_{i=1}^m|\widetilde Q_i^\tau\ra\la \widetilde Q_i^\tau|,
		\end{equation}
		by the definition of $G_\tau$, we have
		\begin{equation*}
			\big(\widetilde Q_i^\tau,\widetilde Q_j^\tau\big)=\delta_{ij},\ \ i,j=1,\cdots,m.
		\end{equation*}
		Using the Taylor's expansion, thus we deduce from \eqref{3.7} that
		\begin{equation}\label{3.8}
			(\tilde Q_1^{\tau}, \cdots ,\tilde Q_m^{\tau})=(Q_1^{\tau}, \cdots ,Q_m^{\tau}) -\frac{1}{2}\big(\sum_{i=2}^mE_{i1}Q_i^{\tau}, \cdots,\sum_{i=1,i\not =j}^m E_{ij}Q_i^{\tau},\cdots,\sum_{i=1}^{m-1}E_{im}Q_i^{\tau}\big)+O({\tau}^{-\infty}).
		\end{equation}
		Following  Proposition \ref{decay}, one can  derive from  \eqref{3.4}, \eqref{3.5} and \eqref{3.8} that for $\tau>0$ large enough,
		\begin{equation*}
			\begin{aligned}
				\inte V(x)|\widetilde Q_i^\tau|^2dx&=
				\inte V(x)\Big[Q_i^{\tau}-\sum_{i\not=j}\frac{1}{2}E_{ji} Q_j^{\tau}+O(\tau^{-\infty})\Big]^2dx\\
				&=\inte V\Big(\f{x}{\tau}+x_0\Big)\varphi^2\Big(\f{x}{\tau}\Big)Q_i^2(x)dx+O(\tau^{-\infty})\\
				&\leq V(x_0)+O(\tau^{-\infty}), \quad 1\leq i \leq m,
			\end{aligned}
		\end{equation*}
		similarly, we get
		\begin{equation*}
			\inte|\nabla \widetilde Q_i^\tau|^2dx=\tau^2\inte|\nabla Q_i|^2dx+O(\tau^{-\infty})<\infty,\ \ 1\leq i \leq m,
		\end{equation*}
		and
		\begin{equation*}
			\begin{aligned}
				\inte (\rho_{\widetilde\ga_\tau}*|x|^{-2})\rho_{\widetilde\ga_\tau}dx
				&=\inte\inte\frac{\Big(\sum\limits_{i=1}^m|\widetilde Q_i^\tau|^2\Big)\Big(\sum\limits_{j=1}^m|\widetilde Q_j^\tau|^2\Big)}{|x-y|^2}dxdy\\
				&=\tau^2\inte (\rho_{\gamma_0}*|x|^{-2})\rho_{\gamma_0}dx+O(\tau^{-\infty}), 
			\end{aligned}
		\end{equation*}
		due to the fact that $Q_i\in H^1(\R^3)$ holds for $1\leq i \leq m$.
		This  implies that for $\tau>0$ large enough, $\widetilde Q_i^\tau\in \mathcal{H}$ holds for $1\leq i \leq m$.
		
		We next estimate each term of $E_K(\widetilde\ga_\tau)$.
		It follows from above calculation that
		\begin{equation}\label{3.9}
			\begin{aligned}
				E_K(\widetilde\ga_\tau)=&\mathrm{Tr}(-\Delta\widetilde\ga_\tau)+\mathrm{Tr}\big(V(x)\widetilde\ga_\tau\big)-K\inte (\rho_{\widetilde\ga_\tau}*|x|^{-2})\rho_{\widetilde\ga_\tau}dx\\
				=&\tau^2\Big[\mathrm{Tr}(-\Delta\ga_0)-K\inte(\rho_{\ga_0}*|x|^{-2})\rho_{\ga_0}dx\Big]
				+MV(x_0)+O(\tau^{-\infty})\\
				=&\frac{\mathcal{K}^{(N)}-K}{\mathcal{K}^{(N)}}\tau^2\mathrm{Tr}(-\Delta\ga_0)+MV(x_0)+O(\tau^{-\infty})\ \ \hbox{as}\ \ \tau\to\infty,
			\end{aligned}
		\end{equation}
		due to the fact that $\ga_0$ is a minimizer of $\mathcal{K}^{(N)}$ with $\|\ga_0\|=1$, and the function $x\mapsto V(x)\varphi^2(x-x_0)$ is bounded and has compact support. 
		Choosing $x_0\in \mathbb{R}^3$ such that $V(x_0)=0$ and $$\tau=(\mathcal{K}^{(N)}-K)^{-\f14}\to\infty \quad\text{ as $K\nearrow \mathcal{K}^{(N)}$},$$  then we deduce that 
		\begin{equation}\label{ktokn}
			e_K(N)\leq E_K(\tilde \gamma_\tau) \leq
			C(\mathcal{K}^{(N)}-K)^\frac12 \to 0\quad \text{as\ } K\nearrow \mathcal{K}^{(N)}. 
		\end{equation}
		If $\Rank \gamma_K\le \lfloor N/2 \rfloor$, from (iv) in Theorem \ref{thm:GNS} we know that $ \mathcal{K}^{(\lfloor N/2 \rfloor)}>\mathcal{K}^{(N)}$.   then we have
		\begin{equation*}
			\begin{split}
				E_K(\gamma_K)&\ge (1-\frac{K}{\mathcal{K}^{(\lfloor N/2 \rfloor)}})\Tr(-\Delta\gamma_K)+\int_{\mathbb{R}^3}V(x)\rho_{\gamma_K} dx\\
                &\ge \frac{\mathcal{K}^{(\lfloor N/2 \rfloor)}-\mathcal{K}^{(N)}}{\mathcal{K}^{(\lfloor N/2 \rfloor)}}\Big(\Tr(-\Delta\gamma_K)+\int_{\mathbb{R}^3}V(x)\rho_{\gamma_K} dx\Big)
				\geq \frac{\mathcal{K}^{(\lfloor N/2 \rfloor)}-\mathcal{K}^{(N)}}{\mathcal{K}^{(\lfloor N/2 \rfloor)}}\mu_V.
			\end{split}
		\end{equation*}
		This contradicts \eqref{ktokn}.
		
		\textbf{3.} Following \eqref{3.9}, one can obtain that for $K> \mathcal{K}^{(N)}$,
		\begin{equation*}
			e_K(N)\leq \lim\limits_{\tau\to\infty} E_K(\widetilde\ga_\tau)
			=\lim\limits_{\tau\to\infty}\Big\{\mathrm{Tr}(-\Delta+V(x))\widetilde\ga_\tau-K\inte (\rho_{\widetilde\ga_\tau}*|x|^{-2})\rho_{\widetilde\ga_\tau}dx\Big\}	=-\infty,
		\end{equation*}
		and hence there is no  minimizer of $e_K(N)$ for $K>\mathcal{K}^{(N)}$.

		As for the case $K=\mathcal{K}^{(N)}$, we choose $x_0\in\R^3$ such that $V(x_0)=0$, 
		it then follows from \eqref{2.1} and \eqref{3.9} that $e_{\mathcal{K}^{(N)}}(N)=0$. We assume that $\ga=\sum_{i=1}^M|u_i\ra\la u_i|$, where $u_i\in\mathcal{H}$ and  $(u_i,u_j)=\delta_{ij}$ for $i,j=1,\cdots,M$, is a minimizer of $e_{\mathcal{K}^{(N)}}(N)$. We then obtain from \eqref{eq:critical} that,
		\begin{equation}\label{3.10}
			\mathrm{Tr}(-\Delta\ga)=\mathcal{K}^{(N)}\inte(\rho_\ga*|x|^{-2})\rho_{\ga}dx,
		\end{equation}
		so that
		\begin{equation}\label{3.11}
			\inte V(x)\rho_{\ga}(x)dx=0.
		\end{equation}
		Following \eqref{3.10}, one can obtain that $u_1(x)$ is the first eigenfunctions of the operator $-\Delta-2\mathcal{K}^{(N)}(\rho_\ga*|x|^{-2})$ in $\R^3$. In view of the fact (cf. \cite[Theorem 11.8]{Lieb}) that $u_1^2(x)>0$ in $\R^3$. Combining with $\lim_{|x|\to\infty}V(x)=\infty$, this indicates that \eqref{3.11} does not hold, therefore there is no minimizer for $e_{\mathcal{K}^{(N)}}(N)$. This completes the proof of Theorem \ref{thm:existence}.
	\end{proof}
	
	\section{Asymptotic behavior of minimizers}\label{sec:4}
	We now turn to the behavior of minimizers as the coupling strength
	approaches the critical value $\mathcal K^{(N)}$ from below.
	The first objective is to identify the natural blow-up scale and to
	show that, after a suitable rescaling and translation, the minimizing
	sequence has a nontrivial limit. We then prove that every
	nonzero limit is an optimizer of the critical finite-rank
	GNS inequality and that the concentration points approach the zero
	set of the trapping potential.
	
	\begin{proof}[Proof of Theorem \ref{thm:concentration}]
		\textbf{(i).}
		If (\ref{3:e}) is false, then there exists a sequence $\{K_n\}$, where $K_n\nearrow \mathcal{K}^{(N)}$ as $n\to\infty$, such that  $\{u_i^{n}(x)\}$ is bounded uniformly in $\mathcal{H}$. By applying  compactness result (see e.g. \cite[Theorem XIII.67]{Reed1}),  there exist a subsequence (still denoted by $\{K_n\}$) of $\{K_n\}$ and $u_0\in \mathcal{H}$ such that
		$$u_i^{n}\overset{n}\rightharpoonup u_i^0 \quad \text{weakly in }\mathcal{H} \quad \text{and}\quad u_i^{n}\overset{n}\rightarrow u_i^0\quad \text{strongly in }L^2(\R^3)\,\quad \text{for}\ 1\leq i \leq N,$$
		and
		\begin{equation*}%
			\rho_{\ga_{n}}\to\rho_{\ga_0}:=\sum_{i=1}^{M}(u_i^{0})^2\ \ \hbox{strongly in} \ \ L^1_{loc}(\R^3)\ \ \hbox{as}\ \ n\to\infty,
		\end{equation*}
		Consequently,
		$$0=e_{\mathcal{K}^{(N)}}(N)\leq E_{\mathcal{K}^{(N)}}(\gamma_0)\leq\lim_{k\to\infty} E_{K_n}(\gamma_{K_n})=\lim_{k\to\infty} e_{K_n}(N)=0,$$
		where $\gamma_0=\sum_{i=1}^M|u_i^0\ra\la u_i^0|, \gamma_{n}=\sum_{i=1}^M|u_i^{n}\ra\la u_i^{n}|.$
		Since Theorem \ref{thm:existence}, $e_K(N)\to0$ as $K_n\nearrow \mathcal{K}^{(N)}$. This then indicates that $\gamma_0$ is a minimizer of $e_{\mathcal{K}^{(N)}}(N)$, which is impossible since Theorem \ref{thm:existence} shows that $e_{\mathcal{K}^{(N)}}(N)$ cannot be attained. So, part \textbf{(i)} is proved.

		\textbf{(ii).}
		Let $\bar w_i^n$ be defined by \eqref{3:2} and as $K_n\nearrow \mathcal{K}^{(N)}$, we have
		\begin{equation*}
			0\leq\mathrm{Tr}(-\Delta\ga_n)-K_n\int_{\R^3}(\rho_{\gamma_n}*|x|^{-2})\rho_{\gamma_n}dx=\varepsilon_n
			^{-2}-K_n\int_{\R^3}(\rho_{\gamma_n}*|x|^{-2})\rho_{\gamma_n}dx\leq
			e_K(N)\rightarrow 0.
		\end{equation*}
		We thus obtain from (\ref{3:e}) that
		\begin{equation*}
			K_n\int_{\R^3}(\rho_{\gamma_n}*|x|^{-2})\rho_{\gamma_n}dx\cdot\varepsilon^{2}_n \to 1 \quad \text{ as }\quad
			K_n\nearrow \mathcal{K}^{(N)}.
		\end{equation*}
		To complete the proof of (ii), we first claim that 
		\begin{equation}\label{1.12}				\liminf_{K_n\nearrow \mathcal{K}^{(N)}}\int_{B_{R_0}(0)}\rho_{\bar\ga_n}(x)dx\geq \eta>0,		
		\end{equation}
		Indeed, we follows from \eqref{3:2} that
		\begin{equation}\label{3.13}
			\mathrm{Tr}(-\Delta\bar\ga_n)= 1\ \ \hbox{and}\ \  \frac{1}{2K_n}\leq\int_{\R^3}(\rho_{\widetilde\gamma_n}*|x|^{-2})\rho_{\widetilde\gamma_n}dx\leq \frac{2}{K_n}\ \ \hbox{as}\ \ K_n\nearrow \mathcal{K}^{(N)}.
		\end{equation}
		We first claim that
		\begin{equation}\label{3.25}
			\mathrm{Tr}(-\Delta\bar\ga_n)=\mathcal{K}^{(N)}\inte(\rho_{\bar\ga_n}* |x|^{-2})\rho_{\bar\ga_n}dx+o(1)\ \ \hbox{as}\ \ K_n\nearrow \mathcal{K}^{(N)},
		\end{equation}
		where $\bar\ga_n$ is defined by \eqref{3:2}.
		Indeed, note that
		$\ga_n=\sum_{i=1}^m|u_i^n\ra\la u_i^n|$ is a minimizer of $e_K(N)$, where $\{u_i^n\}$ satisfies the following system 
		\begin{equation}\label{3.15}
			-\Delta u_i^n+V(x)u_i^n-2K_n\rho_{\ga_n} * |x|^{-2}u_i^n=\mu_i^nu_i^n\ \ \hbox{in}\ \ \R^3,\ \ 1\leq i \leq m.
		\end{equation}
		Here $\rho_{\ga_n}(x)=\sum_{i=1}^m|u_i^n(x)|^2$, and $\mu_i^n$ are the i-first eigenvalues of the operator $-\Delta+V(x)-2K_n\rho_{\ga_n} * |x|^{-2}$ in $\R^3$.
		We hence deduce from \eqref{3.15} and as $K_n\nearrow \mathcal{K}^{(N)}$ that
		\begin{equation}\label{3.17}
			\sum\limits_{i=1}^M\mu_i^n\van^2
			=\van^2E_K(N)-K_n\van^2\inte(\rho_{\ga_n}* |x|^{-2})\rho_{\ga_n}dx
			=-K_n\inte(\rho_{\bar\ga_n}* |x|^{-2})\rho_{\bar\ga_n}dx+o(1).
		\end{equation}

		On the other hand, we obtain from \eqref{3:2} and   \eqref{3.15}  that
		\begin{equation}\label{3.16}
			-\Delta \bar w_i^n+\van^2 V(\van x+\bar x_a)\bar w_i^n-2K\rho_{\bar\ga_n}*|x|^{-2}\bar w_i^n=\mu_i^n\van^2\bar w_i^n\ \ \hbox{in}\ \ \R^3,\ \ 1\leq i \leq M,
		\end{equation}
		which implies that
		\begin{equation*}
			\mathrm{Tr}(-\Delta\bar\ga_n)+\van^2\inte V(\van x+\bar x_a)\rho_{\bar\ga_n}(x)dx-2K_n\inte(\rho_{\bar\ga_n}* |x|^{-2})\rho_{\bar\ga_n}dx
			=\sum\limits_{i=1}^m\mu_i^n\van^2.
		\end{equation*}
		It follows from  $e_K(N) \to 0$ as $K_n\nearrow \mathcal{K}^{(N)}$ that
		\begin{equation*}
			\van^2\inte V(\van x+\bar x_a)\rho_{\bar\ga_n}(x)dx\to0\ \ \hbox{as}\ \ K_n\nearrow \mathcal{K}^{(N)},
		\end{equation*}
		which gives that
		\begin{equation*}
			\mathrm{Tr}(-\Delta\bar\ga_n)-2K_n\inte(\rho_{\bar\ga_n}* |x|^{-2})\rho_{\bar\ga_n}dx
			=\sum\limits_{i=1}^m\mu_i^n\van^2+o(1)\ \ \hbox{as}\ \ K_n\nearrow \mathcal{K}^{(N)}.
		\end{equation*}
		Combine  with \eqref{3.17} yields that
		\begin{equation*}
			1=\mathrm{Tr}(-\Delta\bar\ga_n)=\mathcal{K}^{(N)}\inte(\rho_{\bar\ga_n}* |x|^{-2})\rho_{\bar\ga_n}dx+o(1)\ ,\ \sum\limits_{i=1}^m\mu_i^n\van^2=-1\ \ \hbox{as}\ \ K_n\nearrow \mathcal{K}^{(N)},
		\end{equation*}
		and hence the claim \eqref{3.25} holds true.
		Otherwise, the Hoffmann-Ostenhof  inequality \cite{Hof} gives that
		\begin{equation*}
			\mathrm{Tr}(-\Delta \bar\ga_n)\geq \inte |\nabla\sqrt{\rho_{\bar\ga_n}}|^2dx.
		\end{equation*}
		Thus, we deduce from \eqref{3.13} that the sequence $\{\sqrt{\rho_{\bar\ga_n}}\}$ is bounded uniformly in $H^1(\R^3)$ as $K_n\nearrow \mathcal{K}^{(N)}$.
		
		We next claim that there exist a sequence $\{y_{\van} \}\subset\R^3$,  $R_0>0$ and $\eta>0$ such that
		\begin{equation}\label{3.22}
			\liminf_{K_n\nearrow \mathcal{K}^{(N)}}\int_{B_{R_0}(y_{\van})}\rho_{\bar\ga_n}(x)dx\geq \eta>0.
		\end{equation}
		Indeed, if \eqref{3.22} is not true, then for any $R>0$, there exists a sequence $\{K_n\}$, where $K_n\nearrow \mathcal{K}^{(N)}$ as $n\to\infty$, such that
		\begin{equation*}
			\lim\limits_{n\to\infty}\sup_{y\in\R^3}\int_{B_R(y)}\rho_{\widetilde\ga_{n}}(x) dx=0.
		\end{equation*}
		Since the sequence $\{\sqrt{\rho_{\widetilde\ga_{n}}}\}$ is bounded uniformly in $H^1(\R^3)$ as $n\to\infty$, we derive from \cite[Theorem 1.21]{Willem} that $\rho_{\widetilde\ga_{K_n}}(x)\to0$ strongly in $L^q(\R^3)$ as $n\to\infty$ for $1<q<3$. This is however a contradiction in view of \eqref{3.13}. We therefore obtain that the claim \eqref{3.22} holds true, which further yields that \eqref{1.12} holds true.

		\textbf{(iii).} We first prove that \eqref{03-1} holds true.
		On the contrary, assume that  \eqref{03-1} is not true. Then there exist   a sequence $\{K_n\}$, where $K_n\nearrow \mathcal{K}^{(N)}$ as $n\to\infty$, and a constant $\delta>0$ such that
		\begin{equation*}
			\mathrm{dist}(\bar x_n,\mathcal{A})\geq\delta>0\ \ \hbox{as}\ \ n\to\infty,
		\end{equation*}
		which yields that there exists a constant $C(\delta)>0$ such that
		\begin{equation*}
			V(\bar x_n)\geq C(\delta)>0\ \ \hbox{as}\ \ n\to\infty.
		\end{equation*}
		By Fatou's lemma, we therefore derive from \eqref{1.12} that
		\begin{equation*}\label{3.23}
			\begin{aligned}
				\liminf_{n\to\infty}\inte V(\va_n x+\bar x_n)\rho_{\bar\ga_n}(x)dx
				&\geq\int_{B_{R_0}(0)}\liminf_{n\to\infty}V(\va_n x+\bar x_n)\rho_{\bar\ga_n}(x)dx\geq \f{C(\delta)}{2}\eta>0.
			\end{aligned}
		\end{equation*}
		On the other hand, one can deduce from \eqref{eq:critical} and Theorem \ref{thm:existence} that
		\begin{equation*}
			0<\inte V(\va_n x+\bar x_n)\rho_{\bar\ga_n}(x)dx=\inte V(x)\rho_{\ga_n}(x)dx\leq e_{K_n}(N)\to0\ \ \hbox{as}\ \ n\to\infty,
		\end{equation*}
		which however contradicts with \eqref{3.13}, and hence  \eqref{03-1} holds true.
		We now focus on the proof of \eqref{3-1}. Towards this purpose, 
		
		Let $\{K_n\}$ be any sequence satisfying $K_n\nearrow \mathcal{K}^{(N)}$ as $n\to\infty$. It follows from \eqref{03-1} that there exist a subsequence, still denoted by $\{K_n\}$, of $\{K_n\}$ and a point $x_0\in\mathcal{A}$ such that
		\begin{equation*}\label{03-2}
			\bar x_{n}\to x_0\ \ \hbox{as}\ \ n\to\infty.
		\end{equation*}
		From \eqref{3.22} we see  that there exists a positive operator $0\not =\bar{\gamma}\in \mathcal{S}^1$ such that 
		\begin{equation*}
			\bar \gamma_n\overset{\star}\rightharpoonup \bar{\gamma} \ \text{ in } \mathcal{S}^1\ \text{ and }\  \rho_{\bar \gamma_n} \to \rho_{\bar{\gamma}} \text{ in } L_{\rm loc}^q (\mathbb{R}^3)\ \ \forall \ q\in[1,3),\  \text{ as } n\to\infty.
		\end{equation*}
		Thus we obtain that
		\begin{equation}\label{3.32}
			\mathrm{Tr}(-\Delta\bar\ga_n)=\mathrm{Tr}(-\Delta\bar\ga)
			+\mathrm{Tr}(-\Delta(\bar \gamma_n- \bar{\gamma}))
			+o(1)\ \ \hbox{as}\ \ n\to\infty.
		\end{equation}
		and
		\begin{equation}\label{3.33}
			\int_{\R^3}(\rho_{\bar\gamma_n}*|x|^{-2})\rho_{\bar\gamma_n} dx=\int_{\R^3}(\rho_{\bar\gamma}*|x|^{-2})\rho_{\bar\gamma} dx+\int_{\R^3}(\rho_{(\bar \gamma_n- \bar{\gamma})}*|x|^{-2})\rho_{(\bar \gamma_n- \bar{\gamma})} dx+o(1)
		\end{equation}
		Since $\|\bar\ga\|\leq\liminf_{n\to\infty}\|\bar\ga_n\|=1$ and $\|\eta_n\bar\ga_n\eta_n\|\leq \|\bar\ga_n\|=1$, we obtain from \eqref{eq:critical}, \eqref{3.32} and \eqref{3.33} that
		\begin{equation}\label{3.28}
			\begin{aligned}
				0=&\lim\limits_{n\to\infty}\Big\{\mathrm{Tr}(-\Delta\bar\ga_n)
				-K_n\inte(\rho_{\bar\gamma_n}*|x|^{-2})\rho_{\bar\gamma_n}dx\Big\}\\
				\geq&\mathrm{Tr}(-\Delta\bar\ga)-\mathcal{K}^{(N)}\inte(\rho_{\bar\gamma}*|x|^{-2})\rho_{\bar\gamma}dx\\
                &+\lim\limits_{n\to\infty}\Big\{\mathrm{Tr}(-\Delta(\bar \gamma_n- \bar{\gamma}))-K_n\int_{\R^3}(\rho_{(\bar \gamma_n- \bar{\gamma})}*|x|^{-2})\rho_{(\bar \gamma_n- \bar{\gamma})} dx\Big\}\\
				\geq& \|\bar\ga\|\mathrm{Tr}(-\Delta\bar\ga)-\mathcal{K}^{(N)}\inte(\rho_{\bar\gamma}*|x|^{-2})\rho_{\bar\gamma}dx\geq0,
			\end{aligned}
		\end{equation}
		which implies that $\bar\ga$ is a minimizer of $\mathcal{K}^{(N)}$ and $\|\bar\ga\|=1$. 
		This together with the fact that $\gamma$ is finite rank operator implies that $\gamma\in\mathcal{S}^1$ with $\|\gamma\|=1$     is an optimizer of (\ref{eq:critical}) and 
		\begin{equation}\label{3.29}
			\lim_{k\to\infty}\Big[\text{Tr}\big(-\Delta(\bar \gamma_n- \bar{\gamma})\big)-K_n\int_{\mathbb{R}^3}( \rho_{\bar \gamma_n- \bar{\gamma}})*|x|^{-2}\rho_{\bar \gamma_n- \bar{\gamma}}dx\Big]=0.
		\end{equation}
		We denote   
		$$\gamma=\sum_{i=1}^M|\bar w_i \rangle \langle \bar w_i|,\quad (\bar w_i,\bar w_j)=\delta_{ij},\ i,j=1,2,\cdots,M,  $$ where $ \lfloor N/2 \rfloor+1\le M\le m$  by \eqref{2n<n}, and $\bar w_i$ is $i$-th eigenfunction of $H_{2,\gamma}$.
		
		We claim that 		\begin{equation}\label{3.30}			\Tr\big(-\Delta(\bar \gamma_n- \bar{\gamma})\big)=\Tr(-\Delta\bar \gamma_n)-\Tr(-\Delta\gamma)=o_n(1).		\end{equation}		For otherwise, it follows  from \eqref{3.29} that there exists $C>0$ such that 		$$\Tr\big(-\Delta(\bar \gamma_n- \bar{\gamma})\big), \int_{\mathbb{R}^3}( \rho_{\bar \gamma_n- \bar{\gamma}})*|x|^{-2}\rho_{\bar \gamma_n- \bar{\gamma}}dx\geq C>0. $$		Proceeding the same arguments between \eqref{3.22} and \eqref{3.28},  one can derive that there exist $\{z'_n\}\subset\mathbb{R}^3$ satisfying $|z'_n|\overset{n}\to\infty$ and  $\bar \gamma\in \mathcal{S}^1$ such that, by passing to a subsequence,		\begin{equation*}			( \bar \gamma_n- \bar{\gamma}) (\cdot-z_n')\overset{\star}\rightharpoonup \bar \gamma\not=0,\  \text{in}\ \mathcal{S}^{1}, \text{   as}\ n\to \infty		\end{equation*}		and $\bar \gamma\in \mathcal{S}^1$ with  $\|\bar \gamma\|=1$  is an optimizer of (\ref{eq:critical}).	Then, we can use \eqref{2n<n} to deduce the follow contradiction:		\begin{equation*}			m\equiv \Tr(\bar \gamma_n)=\int_{\mathbb{R}^3}\bar \rho_ndx \ge \int_{\mathbb{R}^3}\rho_\gamma dx+ \int_{\mathbb{R}^3} \rho_{\bar\gamma}dx=\Tr(\gamma)+\Tr(\bar \gamma)> N\ge m.		\end{equation*}		From \eqref{3.33}, \eqref{3.29} and \eqref{3.30}, we deduce that  
		\begin{equation*}
			\lim_{n\to\infty}\Tr(-\Delta\bar \gamma_n)=\Tr(-\Delta\bar\gamma),\quad \lim_{n\to\infty}\int_{\mathbb{R}^3}(\rho_{\bar \gamma_n}*|x|^{-2})\rho_{\bar \gamma_n}dx=\int_{\mathbb{R}^3}(\rho_{\bar \gamma}*|x|^{-2})\rho_{ \bar\gamma}dx.
		\end{equation*}
		
		Finally, if $m=M$, it follows from  $\Tr \bar \gamma_n=\Tr \bar\gamma=m$ that  $\bar \gamma_n\overset{n}\to \bar\gamma $ in $\mathcal{S}^{1}$, and thus
		\begin{equation*}
			\bar w_{i}^n(x)\to \bar w_i(x)\ \text{in}\ L^2(\mathbb{R}^3) \ \text{ for all }i=1,2,\cdots,m, \  \text{   as}\ \ n\to \infty. 
		\end{equation*}
		Then, one can further use \eqref{3.29} to obtain \eqref{3-1}.
	\end{proof}

	\section{Refined blow-up behavior under different trapping potentials}\label{sec:5}
	 
	In this section, under the different assumptions on the potential $V$, we prove Theorem \ref{thm:rates} and investigate how the structure of the potential affects the concentration behavior of minimizers.

	\begin{proof}[Proof of Theorem \ref{thm:rates}]
		\textbf{(i).}
        In order to derive the upper bound of energy, we take the test function $\widetilde{\gamma}_\tau=\sum_{i=1}^M|\widetilde Q_i^\tau\ra\la \widetilde Q_i^\tau|$, where the point $x_0$ in \eqref{3.4} is chosen such that $x_0\in\mathcal{Z}$ defined in \eqref{V2}. Choose sufficiently small $ R>0$ so that
		\begin{equation*}
			V(x)\leq C_1|x-x_0|^p\ \ \hbox{for}\ \ |x-x_0|\leq R.
		\end{equation*}
		We therefore obtain from \eqref{3.4}, \eqref{3.5} and \eqref{3.8} that
		\begin{equation*}
			\begin{aligned}
				\mathrm{Tr}(V\widetilde\ga_\tau)
				&=\sum\limits_{i=1}^M\inte V(x)|\widetilde Q_i^\tau(x)|^2dx\\
				&=\sum\limits_{i=1}^M\inte V\l(\f{x}{\tau}+x_0\r)\varphi^2\l(\f{x}{\tau}\r)Q_i^2(x)dx+O(\tau^{-\infty})\\
				&\leq C_1\tau^{-p}\inte|x|^p\rho_{\ga}(x) dx+O(\tau^{-\infty})\ \ \hbox{as}\ \ \tau\to\infty,
			\end{aligned}
		\end{equation*}
		which then yields from \eqref{eq:critical} that
		\begin{equation*}
			E_K(\gamma)\leq E_K(\widetilde\ga_\tau)
			\leq (\mathcal{K}^{(N)}-K)\tau^2\inte(\rho_{\gamma}*|x|^{-2})\rho_{\gamma}dxdx +C_1\tau^{-p}\inte|x|^p\rho_{\ga}(x) dx+O(\tau^{-\infty})\ \hbox{as}\ \tau\to\infty.
		\end{equation*}
		Setting $\tau=(\mathcal{K}^{(N)}-K)^{-\f{1}{p+2}}>0$ into the above estimate, we get $$E_K(\gamma)\leq M (\mathcal{K}^{(N)}-K)^{\f{p}{p+2}}\ \ \hbox{as}\ \ K\nearrow \mathcal{K}^{(N)}.$$
		Define for $1\leq i \leq m,$
		\begin{equation*}\label{3.44}
			\bar w_i^{n}(x):=\va_{n}^{\f32}u_i^{n}(\va_{n} x+\bar x_{n}),\quad\mbox{and}\quad\bar\ga_{n}:=\sum\limits_{i=1}^m| \bar w_i^{n}\ra\la \bar w_i^{n}|,
		\end{equation*}
		where  $\bar w_i^{n}(x)$ and $\bar x_{n}\in\R^3$ are as in  \eqref{3-1}.
		By direct calculations, we deduce from \eqref{eq:critical} that
		\begin{equation}\label{4.3}
			\begin{aligned}
				e_{K_n}(N)&=\mathrm{Tr}\big(-\Delta+V(x)\big)\ga_n-K_n\inte(\rho_{\gamma_n}*|x|^{-2})\rho_{\gamma_n}dx\\
				&= \frac{\mathcal{K}^{(N)}-K_n}{\mathcal{K}^{(N)}}\varepsilon_n^{-2}+\inte V(\varepsilon_n x+\bar x_n)\rho_{\bar\ga_n}dx.
			\end{aligned}
		\end{equation}
		Next we claim that
		\begin{equation}\label{m=M}
			p_{j_0}=p,\quad \lim_{n\to\infty}\frac{|\bar x_n-x_{j_0}|}{\varepsilon_n}<\infty,\quad
			\lim_{n\to\infty}\inte \rho_{ \bar\gamma_n}dx=\inte \rho_{\gamma}dx.
		\end{equation}
		Indeed, suppose that \eqref{m=M} is false. Then, up to a subsequence, at least one of the following alternatives occurs:
		\begin{equation*}
			p_{j_0}<p,\qquad
			\frac{|\bar x_n-x_{j_0}|}{\varepsilon_n}\to\infty,\qquad\mbox{or}\qquad
			\liminf_{n\to\infty}\left(\inte\rho_{\bar\gamma_n}dx-\inte\rho_{\gamma}dx\right)>0.
		\end{equation*}
		By the local form of $V$ in \eqref{V2}, the local convergence of $\rho_{\bar\gamma_n}$ to $\rho_\gamma$, and the fact that $\rho_\gamma\not\equiv0$, in any one of the above cases we can deduce that, for every fixed constant $C>0$,
		\begin{equation}\label{4.14}
			\liminf_{n\to\infty}\varepsilon_n^{-p}\inte V(\varepsilon_nx+\bar x_n)\rho_{\bar\ga_n}dx\geq C.
		\end{equation}
		Indeed, if $p_{j_0}<p$, the factor $\varepsilon_n^{p_{j_0}-p}$ diverges to $+\infty$. If $p_{j_0}=p$ but $|\bar x_n-x_{j_0}|/\varepsilon_n\to\infty$, then the factor
		\begin{equation*}
			\Big|x+\frac{\bar x_n-x_{j_0}}{\varepsilon_n}\Big|^p
		\end{equation*}
		diverges on every fixed ball carrying a positive amount of the limiting density $\rho_\gamma$. Finally, if the first two alternatives are excluded but
		\begin{equation*}
			\liminf_{n\to\infty}\left(\inte\rho_{\bar\gamma_n}dx-\inte\rho_{\gamma}dx\right)>0,
		\end{equation*}
		then a positive amount of mass of $\rho_{\bar\gamma_n}$ escapes every fixed ball in the rescaled variables, and the same weighted-potential estimate yields \eqref{4.14}.
		On the other hand, by the definition of $\mathcal{K}^{(N)}$ and \eqref{4.14}, for all sufficiently large $n$,
		\begin{equation*}
			e_{K_n}(N)\geq \frac{\mathcal{K}^{(N)}-K_n}{\mathcal{K}^{(N)}}\varepsilon_n^{-2}+C\varepsilon_n^p.
		\end{equation*}
		Minimizing the right-hand side with respect to $\varepsilon_n>0$, we obtain
		\begin{equation*}
			e_{K_n}(N)\geq C_p C^{\frac{2}{p+2}}\big(\mathcal{K}^{(N)}-K_n\big)^{\frac{p}{p+2}},
		\end{equation*}
		where $C_p>0$ is independent of $C$ and $n$. Since $C>0$ can be chosen arbitrarily large, this contradicts the upper bound
		\begin{equation*}
			e_{K_n}(N)\leq M\big(\mathcal{K}^{(N)}-K_n\big)^{\frac{p}{p+2}}
		\end{equation*}
		derived above. Hence \eqref{m=M} holds. In particular,
		\begin{equation*}
			m=\inte\rho_{\bar\gamma_n}dx=\inte\rho_\gamma dx=M.
		\end{equation*}
		It follows from \eqref{m=M} that there exist a subsequence, still denoted by $\{K_n\}$, and a point $\bar x\in\R^3$ such that
		\begin{equation}\label{4.13}
			\lim\limits_{n\to\infty}\f{\bar x_n-x_{j_0}}{\varepsilon_n}=\bar x.
		\end{equation}
		We then obtain from \eqref{4.13} and \eqref{eq:lambda1} that
		\begin{equation}\label{4.5}
			\begin{aligned}
				&\liminf_{n\to\infty}\varepsilon_n^{-p}\inte V(\varepsilon_nx+\bar x_n)\rho_{\bar\ga_n}dx\\
				=&\liminf_{n\to\infty}\inte \f{V(\varepsilon_nx+\bar x_n)}{|\varepsilon_nx+\bar x_n-x_{j_0}|^{p}}\Big|x
				+\f{\bar x_n-x_{j_0}}{\varepsilon_n}\Big|^{p}\rho_{\bar\ga_n}dx\\
				\geq&\al_{j_0}\inte |x+\bar x|^{p}\rho_\ga dx\geq\al_0\inte |x+\bar x|^{p}\rho_\ga dx,
			\end{aligned}
		\end{equation}
		where $\ga=\sum_{i=1}^M|w_i\ra\la w_i|$, and all above identities hold, if and only if  $\al_{j_0}=\al_0$ is as in \eqref{eq:lambda1}.
		We thus deduce from \eqref{4.3} and \eqref{4.5} that
		\begin{equation}\label{4.11}
			\begin{aligned}
				\liminf_{n\to\infty}{e_{K_n}(N)}
				&\geq\frac{\mathcal{K}^{(N)}-K_n}{\mathcal{K}^{(N)}}\varepsilon_n^{-2}+ \al_0\inte |x+\bar x|^{p}\rho_\ga dx\cdot \varepsilon_n^p\\
				&\geq \frac{p+2}{p\mathcal{K}^{(N)}}
				M_1^{-\frac{2}{p+2}}
				(\mathcal{K}^{(N)}-K_n)^{\frac{p}{p+2}}.
			\end{aligned}
		\end{equation}
		For $M_1=\frac{2 }{ \al_0 p\mathcal{K}^{(N)}\inte |x+\bar x|^{p}\rho_\ga dx } $.
		
		
		On the other hand, defining
		\begin{equation*}
			u_i(x)=\varepsilon_n^{-\f32}w_{i}\Big(\f{x-x_m}{\varepsilon_n}-\bar x\Big),\ \ 1\leq i \leq M,
		\end{equation*}
		where $x_m\in\Lambda_1$ is as in \eqref{V2},
		choose $\ga_1=\sum_{i=1}^M|u_i\ra\la u_i|$ as a trail operator of $e_{K_n}(N)$, and assume $\ga=\sum_{i=1}^M|w_i\ra\la w_i|$ 
		and $\mathrm{Tr}(-\Delta\ga)=1$. 
		We then deduce that
		\begin{equation*}
			\begin{aligned}
				e_{K_n}(N)&\leq \mathrm{Tr}\big(-\Delta+V(x)\big)\ga_1-K_n\inte(\rho_{\gamma_1}*|x|^{-2})\rho_{\gamma_1}dx\\
				&=\varepsilon_n^{-2}\Big(\mathrm{Tr}(-\Delta\ga)-K_n\inte(\rho_{\gamma}*|x|^{-2})\rho_{\gamma}dx\Big)
				+\inte V\big(\varepsilon_n(x+\bar x)+x_m\big)\rho_\ga dx\\
				&\leq\frac{\mathcal{K}^{(N)}-K_n}{\mathcal{K}^{(N)}}\varepsilon_n^{-2}+\varepsilon_n^{p}\inte\f{V\big(\varepsilon_n(x+\bar x)+x_m\big)}{|\varepsilon_n(x+\bar x)+x_m-x_m|^p}|x+\bar x|^p\rho_\ga dx,
			\end{aligned}
		\end{equation*}
		which yields that
		\begin{equation}\label{4.12}
			\limsup_{n\to\infty}{e_{K_n}(N)}\leq \frac{p+2}{p\mathcal{K}^{(N)}}
			M_1^{-\frac{2}{p+2}}
			\bigl(\mathcal{K}^{(N)}-K_n\bigr)^{\frac{p}{p+2}}.
		\end{equation}
		We thus conclude from \eqref{4.11} and \eqref{4.12} that
		\begin{equation*}
			\lim\limits_{n\to\infty}e_{K_n}(N)=\frac{p+2}{p\mathcal{K}^{(N)}}
			M_1^{-\frac{2}{p+2}}
			\bigl(\mathcal{K}^{(N)}-K_n\bigr)^{\frac{p}{p+2}},
		\end{equation*}
		and 
		\begin{equation*}
			\varepsilon_n=(1+o(1))[M_1( \mathcal{K}^{(N)}-K_n)]^{\frac{1}{p+2}}.
		\end{equation*}
		Together with  \eqref{4.5}, this further implies that $\al_{j_0}=\al_0$. Finally, from \eqref{m=M} we know that 
		$$m\equiv\int_{\mathbb{R}^3}\rho_{\bar\gamma_n}dx=\int_{\mathbb{R}^3}\rho_{\gamma}dx=M$$ and  then \eqref{3-1}  follows.

			\textbf{(ii).}
			The upper bound estimate for energy is similar to \eqref{V2}, \begin{equation*}
				e_K(N)\leq E_K(\widetilde\ga_\tau)
				\leq (\mathcal{K}^{(N)}-K)\tau^2\inte(\rho_{\gamma}*|x|^{-2})\rho_{\gamma}dx+ C_1\tau^{-2}\inte|x|^2\rho_{\ga}(x) dx+O(\tau^{-\infty})\ \hbox{as}\ \tau\to\infty.
			\end{equation*}
			Setting $\tau=(\mathcal{K}^{(N)}-K)^{-\f{1}{4}}>0$, we get $$E_K(\gamma)\leq M (\mathcal{K}^{(N)}-K)^{\f{1}{2}}\ \ \hbox{as}\ \ K\nearrow \mathcal{K}^{(N)}.$$
            Similar to \eqref{4.3}, we have
			\begin{equation*}
				e_{K_k}(N)= \frac{\mathcal{K}^{(N)}-K_k}{\mathcal{K}^{(N)}}\varepsilon_{k}^{-2}+\inte V(\varepsilon_{k} x+\bar x_k)\rho_{\bar\ga_k}dx.
			\end{equation*}
			
			Similarly, our claim \eqref{m=M} holds.
			We assume that there exists a subsequence still denoted by $\{K_k\}$ such that
			\begin{equation*}
				\frac{|\bar x_k|-|x_A|}{\varepsilon_k} \to C_0 \quad \text{as}\quad
				k\rightarrow\infty \end{equation*} for some constant $C_0$. Since $\rho_{\bar\ga_k}$
			decays exponentially as
			$|x|\to\infty$, we then deduce  that
			\begin{align}\label{v4a}
				\lim_{k \to \infty} \frac{1}{\varepsilon_k^{2}}\Tr(V\ga_k)\geq
				&\lim_{k \to \infty} \frac{1}{\varepsilon_k^{2}} \inte \big(|\varepsilon_k x+\bar x_k
				|-|x_A|\big)^2 \rho_{\bar\ga_k} dx \nonumber\\
				&=\lim_{k \to \infty} \inte \Big(\frac{|\varepsilon_k x+\bar
					x_k |-|\bar x_k|}{\varepsilon_k}+\frac{|\bar x_k|-|x_A|}{\varepsilon_k}\Big)^2
				\rho_{\bar\ga_k}dx \nonumber\\
				&=\inte \Big(\frac{x_A\cdot
					x}{|x_A|}+C_0\Big)^2 \rho_{\bar\ga}dx ,
			\end{align}
			without loss of generality, we assume that $ x_A=(A, 0, 0)$, and define $f(t):=\inte ( x_1+t)^2 \rho_{\bar\ga}dx$, $f(t)\in C(\R)$, there exists a constant $M_2=\inte x_1^2 \rho_{\bar\ga}dx-\frac{(\inte  x_1 \rho_{\bar\ga}dx)^2}{\inte \rho_{\bar\ga}dx}$, independent of $K_k$, such that $M_2=\inf_{t\in\R}f(t)$, then
			\[
			\int_{\R^3}V(\varepsilon_k x+\bar x_k)\rho_{\bar\ga_k}dx \geq
			M_2\varepsilon_k^2 \quad\text{as}\quad K_k\nearrow \mathcal{K}^{(N)} .\]      
			We thus deduce from \eqref{4.3} and \eqref{v4a} that
			\begin{equation*}
				\begin{aligned}
					\liminf_{k\to\infty}{e_{K_k}(N)}
					&\geq\frac{\mathcal{K}^{(N)}-K_k}{\mathcal{K}^{(N)}}\varepsilon_{k}^{-2}+ M_2\varepsilon_k^2\geq 2\sqrt{\frac{ M_2 }{\mathcal{K}^{(N)}}}(\mathcal{K}^{(N)}-K_k)^{\frac{1}{2}}\quad\text{as}\quad K_k\nearrow \mathcal{K}^{(N)}.
				\end{aligned}
			\end{equation*}
			
			Similar with the situation of \eqref{V2}, we have
			\begin{equation*}\label{4.12v4}
				\liminf_{k\to\infty}e_{K_k}(N)=2\sqrt{\frac{ M_2 }{\mathcal{K}^{(N)}}}(\mathcal{K}^{(N)}-K_k)^{\frac{1}{2}}\quad\text{as}\quad K_k\nearrow \mathcal{K}^{(N)} .\end{equation*}
			and
			\begin{equation*}
				\varepsilon_{k}=(1+o(1))(\frac{\mathcal{K}^{(N)}-K_k}{\mathcal{K}^{(N)}M_2})^{\frac{1}{4}}
			\end{equation*}
			Moreover, we also get $\int_{\mathbb{R}^3}\rho_{\bar\gamma}dx=M$ and  then \eqref{3-1}  follows.
			
			\textbf{(iii).}
			First, we consider upper bound estimates for energy. Let $\varphi(x)\in
			C^\infty_0(\R^3)$ be
			a nonnegative smooth cut-off function such that $\varphi(x)=1$ if $|x|\leq1$
			and $\varphi(x)=0$ if $|x|\geq2$. Choose $\Omega_i\in \Lambda_3$ and  $x_0\in\Omega_i$ for some $i\in\{1,\cdots ,m\}$ such that
			\begin{equation*}
				{\rm dist} (x_0, \partial \Omega_i)=R_i=R,
				\end{equation*}
			where  $R$ and $\Lambda_3$ are defined by \eqref{V4}.
			Similar with Theorem \ref{thm:existence},
			for any $x_0\in\R^3$ and $\tau>0$, define
			\begin{equation*}\label{3.4r}
				Q_i^{\tau R}(x):=A_i^{\tau R}\tau^{\f32}\varphi(\frac{x-x_0}{R})Q_i\big(\tau(x-x_0)\big),\ \ 1\leq i \leq N,\ \ \mbox{and}\ \  \ga_{\tau R}^{(m)}:=\sum\limits_{i=1}^m|Q_i^{\tau R}\ra\la Q_i^{\tau R}|,
			\end{equation*}
			where $Q_i$ is chosen such that $\mathrm{Tr}(-\Delta \ga)=1$, and $A_i^{\tau R}>0$ is chosen such that $\inte |Q_i^{\tau R}|^2dx=1$ for $1\leq i \leq N$.
			We therefore obtain that
			\begin{equation}\label{3.5r}
				(A_i^{\tau R})^2=1+Ce^{-2\sqrt{|\mu_i|}R\tau},\ \ 1\leq i \leq N,\ \ \hbox{and}\ \ a_{ij}:=(Q_i^{\tau R},Q_j^{\tau R})=Ce^{-(\sqrt{|\mu_i|}+\sqrt{|\mu_j|})R\tau}\ \ \hbox{as}\ \ \tau\to\infty,
			\end{equation}
			Similar with $G_\tau$, for $\tau>0$ large enough, we define
			\begin{equation*}
				G_{\tau R}:=\left( \begin{matrix}
					1 & \cdots & (Q_1^{\tau R}, Q_m^{\tau R})_{L^2} \\
					\vdots &  \ddots  & \vdots\\
					(Q_m^{\tau R}, Q_1^{\tau R})_{L^2} &  \cdots & 1
				\end{matrix}\right)_{m\times m}, 
			\end{equation*}
			and
			\begin{equation*}
				\left(
				\widetilde Q_1^{\tau R}, \cdots,  \widetilde Q_m^{\tau R} 
				\right)
				:=\left(
				Q_1^{\tau R} ,\cdots, Q_m^{\tau R} 
				\right)
				G_{\tau R}^{-\f12},\quad \widetilde\ga_{\tau R}:=\sum\limits_{i=1}^m|\widetilde Q_i^{\tau R}\ra\la \widetilde Q_i^{\tau R}|.
			\end{equation*}
			Combining with \eqref{3.5r}, we obtain that as $\tau\to\infty$
			\begin{equation*}
				\mathrm{Tr}(-\Delta\widetilde \ga)\leq(1+Ce^{-2\sqrt{|\mu_M|}R\tau})^2\tau^2(\mathrm{Tr}(-\Delta\ga)+C e^{-2\sqrt{|\mu_M|}R\tau})+C\tau^2 e^{-2\sqrt{|\mu_M|}R\tau}
				,
			\end{equation*}
			and
			\begin{equation*}
				\int_{\R^3}(\rho_{\widetilde\gamma}*|x|^{-2})\rho_{\widetilde\gamma}dx\geq (1+Ce^{-2\sqrt{|\mu_1|}R\tau})^2  \tau^2(\int_{\R^3}(\rho_{\gamma}*|x|^{-2})\rho_{\gamma}dx-C e^{-3\sqrt{|\mu_M|}R\tau })+C\tau^2 e^{-2\sqrt{|\mu_1|}R\tau}.
			\end{equation*}
			where $\widetilde{\gamma}=\sum_{i=1}^M|\widetilde Q_i^{\tau R}\ra\la \widetilde Q_i^{\tau R}|$. Using \eqref{eq:critical}, \eqref{3.5r}, and the two estimates above, we obtain
			\begin{equation}\label{2:limt2}
				\begin{split}
					&\mathrm{Tr}(-\Delta\widetilde \ga)-K\int_{\R^3}(\rho_{\widetilde\gamma}*|x|^{-2})\rho_{\widetilde\gamma}dx\\
					&\leq \tau^2(\mathrm{Tr}(-\Delta\ga)-K\int_{\R^3}(\rho_{\gamma}*|x|^{-2})\rho_{\gamma}dx)+C\tau^2 e^{-2\sqrt{|\mu_M |}R\tau} \\
					&=\tau^2(1-\frac{K}{\mathcal{K}^{(N)}})\mathrm{Tr}(-\Delta\ga)+C\tau^2e^{-2\sqrt{|\mu_M|}R\tau}
					\quad\text{as}\quad \tau\to\infty\,.		
				\end{split}
			\end{equation}
			On the other hand, since
			$$V\big(\frac{x}{\tau}+x_0\big)\leq \sup_{y\in B_{2R}(x_0)}V(y) \quad \text{ for all }\tau>0 \text{ and }x\in B_{2\tau R}(0)\,,$$
			By the choice of $x_0$, we therefore have
			$$V\big(\frac{x}{\tau}+x_0\big)\equiv0\quad \text{ for all } x\in B_{\tau R}(0)\,.$$
			It then follows  from the exponential decay of $Q_i$ that
			\begin{equation*}
				\begin{split}
					\mathrm{Tr}(V(x) \widetilde{\gamma})&\leq\sum_{i=1}^M(A_{i}^{\tau R})^2\int_{B_{2\tau R}(0)}V\big(\frac{x}{\tau}+x_0\big)
					\varphi^2\big(\frac{x}{\tau R}\big)Q_i^2(x)dx\\
					&\leq \sum_{i=1}^MC\int_{B_{2\tau R}(0)\setminus B_{\tau R}(0)}|Q_i(x)|^2dx\leq Ce^{-2\sqrt{|\mu_M|}R\tau}\quad\text{as}\quad \tau\to\infty\,.
				\end{split}
			\end{equation*}
			Combining with (\ref{2:limt2}), this implies that
			\begin{equation*}
				0\le E_K(\gamma)\leq \tau^2(1-\frac{K}{\mathcal{K}^{(N)}})+C\tau^2e^{-2\sqrt{|\mu_M|}R\tau}\quad\text{as}\quad \tau\to\infty\,.
			\end{equation*}
			By setting $\tau=\frac{1}{2R\sqrt{|\mu_M| }}(|\ln(\mathcal{K}^{(N)}-K)|+\ln|\ln(\mathcal{K}^{(N)}-K)|)$, we then conclude from the above that
			\begin{equation}\label{upper}
				\begin{split}
					0\le E_K(\gamma)&\leq  \frac{(1+o(1))}{4|\mu_M|R^2 \mathcal{K}^{(N)}}(\mathcal{K}^{(N)}-K)\big(|\ln(\mathcal{K}^{(N)}-K)|\big)^2+\frac{C}{2R}(\mathcal{K}^{(N)}-K)|\ln(\mathcal{K}^{(N)}-K)|\\
					&=\frac{1+o(1)}{4|\mu_M|R^2 \mathcal{K}^{(N)}}(\mathcal{K}^{(N)}-K)\big(|\ln(\mathcal{K}^{(N)}-K)|\big)^2 \quad \mbox{as}\quad K\nearrow \mathcal{K}^{(N)},
				\end{split}
			\end{equation}
			Next, we consider lower bound estimates for energy. For any  sequence
			$\{K_k\}$ with $K_k\nearrow \mathcal{K}^{(N)}$ as $k\to\infty$, we still denote $u_k:=u_{K_k}$  a nonnegative minimizer of $e_{K_k}(N)$.  It then follows from Theorem~\ref{thm:concentration} that
			there exists a subsequence of $\{K_k\}$, still denoted by $\{K_k\}$,  such that each $u_i^k$ has a
			point $\bar x_n$, which satisfies $\bar x_n\overset{k}{\to} x_0$ for some $x_0\in\bar\Omega$.
			Moreover, we have
			\begin{equation*}
				\varepsilon_k:=\Big(\mathrm{Tr}(-\Delta\ga_a)\Big)^{-\frac{1}{2}}\to0\quad \text{ as }\ k\to\infty\,,\end{equation*}
			and
			\begin{equation}\label{4.28}
				\bar w_i^k(x):=\varepsilon_k^{\f32} u_i^{k}\big(\varepsilon_k x+\bar x_n
				\big) \to Q_i(x)\ \text{ strongly in }\ H^1(\R^3)\ \text{as}\ k\to\infty.
			\end{equation}
			
			Since $x_0\in\bar\Omega$,  without loss of generality we may assume that $x_0\in\bar\Omega_{j_0}$ for some $j_0\in\{1,\cdots,m\}$. We first prove that (\ref{1:R0}) holds. Indeed, since $\bar x_n\to x_0$ as $k\to\infty$, where $x_0\in\bar\Omega_{j_0}$, we  always have
			\begin{equation*}\label{3:rel}
				R_0:=\lim_{k\to\infty}{\rm dist}(\bar x_n, \partial \Omega_{j_0})\leq R_{j_0}\leq R\,.\end{equation*}
			For any fixed $0<\delta<\min\{\frac{R_{j_0}}{5},\frac{d}{5}\}$, where $d>0$ is defined by
			\begin{equation*}
				d=\inf\big\{{\rm dist}(\Omega_i,\Omega_j);\, i\neq j\,\text{ and } \, i,j=1,\cdots, m\big\}\,.
			\end{equation*}
			Since $\bar x_n\to x_0\in\bar \Omega_{j_0}$ as $k\to\infty$,  we obtain  $\bar x_n\in \Omega_{j_0}^\delta$ for large $k$ and
			\begin{equation*}
				{\rm dist}(x_0, \partial \Omega_{j_0})=R_0,
				\end{equation*}
			where $\Omega_{j_0}^\delta:=\{x\in\R^3: {\rm dist}(x,\Omega_{j_0})<\delta\}$.
			Choosing $r_1=R_0+3\delta, \ r_2=R_0+4\delta,$
			By the definition of $R_0$, we then have
			$$\Big|\big(B_{r_2}(x_0)\setminus B_{r_1}(x_0)\big)\cap (\Omega^\delta)^c\Big|>0\,.$$
			Recalling that $\bar x_n\to x_0$ as $k\to\infty$, then for large $k>0$,
			\begin{equation*}
				\Big|\big(B_{r_2}(\bar x_n)\setminus B_{r_1}(\bar x_n)\big)\cap (\Omega^\delta)^c\Big|\geq\frac{1}{2}\Big|\big(B_{r_2}(x_0)\setminus B_{r_1}(x_0)\big)\cap (\Omega^\delta)^c\Big|>0\,.
				\end{equation*}
			In view of (\ref{3.22})  and  $\bar w_k(x)$ satisfies (\ref{3.16}) and (\ref{4.28}),   we can apply  Proposition \ref{decay} to see that 
			\begin{equation*}
				\bar w_k(x)\geq C|x|^{-1} e^{-\sqrt{|\mu_1|}|x|}\quad \text{in }\ B_{\frac{r_2}{\varepsilon_k}}(0)\setminus B_{\frac{r_1}{\varepsilon_k}}(0)
				\end{equation*}
			holds for large $k>0$.
			Moreover, note that there exists $C=C(\delta)>0$ such that
			\begin{equation*}
				V(x)\geq C\quad \text{for}\quad x\in\big(B_{r_2}(\bar x_n)\setminus B_{r_1}(\bar x_n)\big)\cap (\Omega^\delta)^c\,.
			\end{equation*}
			with the above geometric estimate, the lower decay bound, and the positivity of $V$ on
			$(\Omega^\delta)^c$, we obtain
			\begin{equation}\label{vxxj}
				\begin{split}
					&\mathrm{Tr}(V(x) \gamma)\geq \sum_{i=1}^MC(\delta)\int_{\big(B_{r_2}(\bar x_n)\setminus B_{r_1}(\bar x_n)\big)\cap (\Omega^\delta)^c}u_k^2dx\\
					&=\sum_{i=1}^MC(\delta)\int_{\big(B_{\frac{r_2}{\varepsilon_k}}(0)\setminus B_{\frac{r_1}{\varepsilon_k}}(0)\big)\cap (\Omega_{\varepsilon_k}^\delta)^c}\bar w_k^2dx\geq C(\delta,\sigma)e^{-2\sqrt{|\mu_1|}\frac{r_2}{\varepsilon_k}}\,,
				\end{split}
			\end{equation}
			where $\Omega_{\varepsilon_k}^\delta=\{x\in\R^3: \varepsilon_k x+\bar x_n\in \Omega^\delta\}$.
			It then follows from the inequality (\ref{eq:critical}) and (\ref{vxxj}) that
			\begin{equation}\label{3:fin}
				\begin{split}
					e_{K_k}(N)=E_{K_k}(\bar\gamma_k)
					&\geq
					(\frac{\mathcal{K}^{(N)}-K_k}{\mathcal{K}^{(N)}})\mathrm{Tr}(-\Delta\ga_a)+\mathrm{Tr}(V(x)  \gamma)\\
					&\geq(\frac{\mathcal{K}^{(N)}-K_k}{\mathcal{K}^{(N)}})\varepsilon_k^{-2}+Ce^{-2\sqrt{|\mu_1|}\frac{r_2}{\varepsilon_k}}\quad \text{as}\quad K_k\nearrow {\mathcal{K}^{(N)}}\,.
				\end{split}
			\end{equation}
			One can check that the function
			\begin{equation*}
				f_\theta(s)=\frac{\mathcal{K}^{(N)}-K}{\mathcal{K}^{(N)}}s^2+Ce^{-\theta s}, \quad\theta>0,\quad s\in(0,\infty)\,,
			\end{equation*}
			has a unique minimum point  $s_\theta$   satisfying
			$$s_\theta=\frac{1}{\theta}\big[|\ln(\mathcal{K}^{(N)}-K)|-\ln|\ln(\mathcal{K}^{(N)}-K)|+C(\theta)\big]\quad \text{as}\quad K\nearrow \mathcal{K}^{(N)}\,.$$
			Thus,
			$$f_\theta(s) \ge f_\theta(s_\theta)= \frac{1+o(1)}{{\mathcal{K}^{(N)}}\theta^2}(\mathcal{K}^{(N)}-K)|\ln(\mathcal{K}^{(N)}-K)|^2\quad \text{as}\quad K\nearrow \mathcal{K}^{(N)}\,.$$
			This estimate and (\ref{3:fin}) imply that
			\begin{equation}\label{4.33}
				e_{K_k}(N)\geq \frac{1+o(1)}{[2\sqrt{|\mu_1|}r_2]^2 {\mathcal{K}^{(N)}}}(\mathcal{K}^{(N)}-K_k)\big|\ln(\mathcal{K}^{(N)}-K_k)\big|^2\quad \text{as}\quad K_k\nearrow {\mathcal{K}^{(N)}}\,.
			\end{equation}
			comparing with  (\ref{upper}), we then have
			$$\sqrt{|\mu_1|}r_2=\sqrt{|\mu_1|}(R_0+4\delta)\geq \sqrt{|\mu_M|} R,\quad \text{for any}\quad 0<\delta<\min\big\{\frac{R_{j_0}}{5},\frac{d}{5}\big\}\,. $$
			This yields  that $R_0\geq \sqrt{\frac{\mu_M}{\mu_1}}R$. We therefore conclude   (\ref{1:R0}).
			Similar with the situation of \eqref{V2}, we obtain from (\ref{upper}) and (\ref{4.33}) that there exist a subsequence of $\{K_n\}$ satisfying $R_0={\rm dist}(x_0, \partial \Omega_{j_0})\geq \sqrt{\frac{\mu_M}{\mu_1}}R$ such that
			\begin{equation*}
				\bar x_n\stackrel{n}{\rightarrow}
				x_0\ \ \hbox{and}\ \ \bar w_i^{n}(x):=\va_{n}^{\f32}u_i^{n}(\va_{n} x+\bar x_n)\stackrel{n}{\rightarrow} \bar w_i(x)
			\end{equation*}
			strongly in $H^1(\R^3)$, for all $i=1, \cdots, M$.

			\end{proof}
			
			\appendix
			\section{Regularity and decay of optimizers for GNS inequality}\label{sec:A}
			
			The decay analysis below is motivated in part by the corresponding
			study of nonlinear Choquard equations in \cite{jfa2013}.
			Since a minimizer of the GNS inequality satisfies a finite system of partial differential equations,
			we provide a direct argument adapted to our setting.
			Throughout this appendix, let
			\[
			\gamma=\sum_{j=1}^{M}k_j|u_j\rangle\langle u_j|,
			\qquad
			k_j>0,
			\qquad
			\rho_\gamma=\sum_{j=1}^{M}k_j|u_j|^2,
			\]
			be a finite-rank optimizer of the Hartree GNS inequality, and assume that
			the occupied orbitals satisfy
			\begin{equation}
				-\Delta u_i-\frac{4}{\alpha}
				\bigl(\rho_\gamma*|x|^{-\alpha}\bigr)u_i
				=\mu_i u_i
				\qquad\text{in }\R^3,
				\qquad 1\le i\le M.
				\label{eq:effectiveA}
			\end{equation}
			We write
			\[
			W_\alpha(x):=
			\bigl(\rho_\gamma*|x|^{-\alpha}\bigr)(x).
			\]
			
			We first recall the Hardy--Littlewood--Sobolev estimate in the notation
			used in this paper. Since the kernel $|x|^{-\alpha}$ is, up to a
			multiplicative constant, the Riesz potential of order $3-\alpha$, if
			$1<p<3/(3-\alpha)$ and
			\[
			\frac1q=\frac1p-\frac{3-\alpha}{3},
			\]
			then
			\begin{equation}
				\bigl\||x|^{-\alpha}*f\bigr\|_{L^q(\R^3)}
				\le C_{\alpha,p}\|f\|_{L^p(\R^3)}.
				\label{eqHLS}
			\end{equation}
			
			\begin{proposition}\label{propositionRegularity}
				Let $0<\alpha<3$ and suppose that
				$u_1,\cdots,u_M\in H^1(\R^3)$ satisfy \eqref{eq:effectiveA}.
				Then, for every $i=1,\cdots,M$,
				\[
				u_i\in L^\infty(\R^3)\cap W^{2,p}_{\mathrm{loc}}(\R^3)
				\qquad\text{for every }1<p<\infty.
				\]
				Moreover, $u_i$ is uniformly continuous and
				\[
				\lim_{|x|\to\infty}u_i(x)=0.
				\]
				In particular,
				\[
				\rho_\gamma\in L^1(\R^3)\cap L^\infty(\R^3),
				\qquad
				W_\alpha\in L^\infty(\R^3).
				\]
			\end{proposition}
			
			\begin{proof}
				Since $u_i\in H^1(\R^3)$, the Sobolev embedding gives
				$u_i\in L^s(\R^3)$ for every $2\le s\le6$. Hence
				\[
				\rho_\gamma\in L^p(\R^3)
				\qquad\text{for every }1\le p\le3.
				\]
				We claim that one can choose $p\in(1,3]$ in \eqref{eqHLS} so that the
				corresponding exponent $q$ satisfies $q>3$.
				
				Indeed, if $0<\alpha\le2$, choose
				$p<3/(3-\alpha)$ sufficiently close to $3/(3-\alpha)$.
				Then $q$ can be taken arbitrarily large. If $2<\alpha<3$, we may take
				$p=3$, which gives
				\[
				q=\frac{3}{\alpha-2}>3.
				\]
				Therefore, for some $q>3$,
				\[
				W_\alpha\in L^q(\R^3).
				\]
				
				Fix such a $q$. Since $u_i\in L^6(\R^3)$,
				\[
				W_\alpha u_i\in L^r(\R^3),
				\qquad
				\frac1r=\frac1q+\frac16<\frac12,
				\]
				and hence $r>2$. Also $u_i\in L^r(\R^3)$ by interpolation between
				$L^2$ and $L^6$. Rewriting \eqref{eq:effectiveA} as
				\[
				(-\Delta+1)u_i
				=
				(\mu_i+1)u_i+\frac4\alpha W_\alpha u_i,
				\]
				the global Calder\'on--Zygmund estimate \cite{pde} yields
				$u_i\in W^{2,r}(\R^3)$. Since $r>3/2$, Sobolev embedding implies
				$u_i\in L^\infty(\R^3)$.
				
				Consequently,
				\[
				\rho_\gamma\in L^1(\R^3)\cap L^\infty(\R^3).
				\]
				Splitting the kernel into its local and far-field parts,
				\[
				|x|^{-\alpha}
				=
				|x|^{-\alpha}\mathbf 1_{\{|x|<1\}}
				+
				|x|^{-\alpha}\mathbf 1_{\{|x|\ge1\}},
				\]
				we obtain
				\[
				\|W_\alpha\|_\infty
				\le
				\|\rho_\gamma\|_\infty
				\int_{|x|<1}|x|^{-\alpha}\,dx
				+\|\rho_\gamma\|_1<\infty,
				\]
				because $\alpha<3$.
				
				It now follows from \eqref{eq:effectiveA} and standard local elliptic
				estimates that
				$u_i\in W^{2,p}_{\mathrm{loc}}(\R^3)$ for every $1<p<\infty$.
				In particular, choosing $p>3$ gives a uniform local
				$C^{1,\beta}$ bound for some $\beta\in(0,1)$, and hence $u_i$ is
				uniformly continuous. Since $u_i\in L^2(\R^3)$, uniform continuity
				implies
				\[
				u_i(x)\rightarrow0
				\qquad\text{as }|x|\to\infty.
				\]
			\end{proof}
			
			\begin{lemma}\label{lemma x0}
				Under the assumptions of Proposition~\ref{propositionRegularity},
				\begin{equation}
					\lim_{|x|\to\infty}W_\alpha(x)=0.
					\label{eq:Wto0}
				\end{equation}
			\end{lemma}
			
			\begin{proof}
				By Proposition~\ref{propositionRegularity},
				$\rho_\gamma$ is bounded, integrable, continuous, and
				$\rho_\gamma(x)\to0$ as $|x|\to\infty$.
				Write
				\[
				W_\alpha(x)=I_1(x)+I_2(x)+I_3(x),
				\]
				where, for a fixed $R>1$,
				\begin{align*}
					I_1(x)
					&:=
					\int_{|x-y|<1}\frac{\rho_\gamma(y)}{|x-y|^\alpha}\,dy,\\
					I_2(x)
					&:=
					\int_{\substack{|x-y|\ge1\\ |y|\le R}}
					\frac{\rho_\gamma(y)}{|x-y|^\alpha}\,dy,\\
					I_3(x)
					&:=
					\int_{\substack{|x-y|\ge1\\ |y|>R}}
					\frac{\rho_\gamma(y)}{|x-y|^\alpha}\,dy.
				\end{align*}
				For the first term,
				\[
				I_1(x)
				\le
				\sup_{|y-x|<1}\rho_\gamma(y)
				\int_{|z|<1}|z|^{-\alpha}\,dz
				\rightarrow0
				\]
				as $|x|\to\infty$. For $|x|>2R$,
				\[
				I_2(x)
				\le
				(|x|-R)^{-\alpha}\|\rho_\gamma\|_1
				\rightarrow0.
				\]
				Finally,
				\[
				I_3(x)
				\le
				\int_{|y|>R}\rho_\gamma(y)\,dy.
				\]
				Since $\rho_\gamma\in L^1(\R^3)$, the last quantity can be made
				arbitrarily small by choosing $R$ sufficiently large. This proves
				\eqref{eq:Wto0}.
			\end{proof}
			
			\begin{proposition}\label{decay}
				Assume the hypotheses of Proposition~\ref{propositionRegularity} and
				let $\mu_i<0$. Set
				\[
				\kappa_i:=\sqrt{-\mu_i}.
				\]
				Then, for every $0<\sigma<\kappa_i$, there exist
				$C_{\sigma,i}>0$ and $R_{\sigma,i}>0$ such that
				\begin{equation}
					|u_i(x)|+|\nabla u_i(x)|
					\le
					C_{\sigma,i}e^{-\sigma|x|}
					\qquad
					\text{for }|x|\ge R_{\sigma,i}.
					\label{eq:rough-decay}
				\end{equation}
				
				If $\alpha=2$ and all occupied eigenvalues
				$\mu_1,\cdots,\mu_M$ are negative, then the estimate can be sharpened to
				\begin{equation}
					|u_i(x)|+|\nabla u_i(x)|
					\le
					C_i|x|^{-1}e^{-\kappa_i|x|}
					\qquad
					\text{for }|x|\ge R_i,
					\qquad 1\le i\le M.
					\label{eq:sharp-upper}
				\end{equation}
				Furthermore, if $u_1$ is chosen as the positive first eigenfunction of
				the limiting effective operator, then
				\begin{equation}
					u_1(x)
					\ge
					c_1|x|^{-1}e^{-\kappa_1|x|}
					\qquad
					\text{for }|x|\ge R_1.
					\label{eq:sharp-lower}
				\end{equation}
			\end{proposition}
			
			\begin{proof}
				We first prove \eqref{eq:rough-decay}. By Kato's inequality \cite{Reed} applied to
				\eqref{eq:effectiveA},
				\[
				(-\Delta+\kappa_i^2)|u_i|
				\le
				\frac4\alpha W_\alpha |u_i|
				\]
				in the sense of distributions. By Lemma~\ref{lemma x0}, for every
				$0<\sigma<\kappa_i$ there exists $R>0$ such that
				\[
				\frac4\alpha W_\alpha(x)
				\le
				\kappa_i^2-\sigma^2
				\qquad\text{for }|x|\ge R.
				\]
				Hence
				\begin{equation*}
					-\Delta|u_i|+\sigma^2|u_i|\le0
					\qquad\text{in }\R^3\setminus B_R.
				\end{equation*}
				The function
				\[
				v(x)=A e^{-\sigma(|x|-R)}
				\]
				is a positive supersolution of
				$-\Delta+\sigma^2$ in $\R^3\setminus B_R$ for $R$ sufficiently
				large. Choosing $A$ so that $v\ge |u_i|$ on $\partial B_R$ and using
				the exterior comparison principle, together with
				$u_i(x)\to0$ as $|x|\to\infty$, gives
				\[
				|u_i(x)|\le C_{\sigma,i}e^{-\sigma|x|}.
				\]
				The gradient estimate in \eqref{eq:rough-decay} follows from the
				standard interior $W^{2,p}$ estimate applied to
				\[
				-\Delta u_i
				=
				-\kappa_i^2u_i+\frac4\alpha W_\alpha u_i
				\]
				on unit balls centered at $x$, using that $W_\alpha$ is bounded.
				
				We now assume $\alpha=2$ and
				$\mu_1,\cdots,\mu_M<0$. The first part of the proposition yields,
				for some $\sigma>0$,
				\[
				\rho_\gamma(x)\le Ce^{-2\sigma|x|}
				\qquad\text{for }|x|\gg1.
				\]
				We claim that
				\begin{equation}
					W_2(x)\le \frac{C}{|x|^2}
					\qquad\text{for }|x|\gg1.
					\label{eq:W2-rate}
				\end{equation}
				Indeed, if $r=|x|$ is large, then
				\[
				\int_{|y|\le r/2}\frac{\rho_\gamma(y)}{|x-y|^2}\,dy
				\le
				\frac{4}{r^2}\|\rho_\gamma\|_1.
				\]
				For the complementary region $|y|>r/2$, split further into
				$|x-y|<1$ and $|x-y|\ge1$. The exponential bound on $\rho_\gamma$
				gives
				\[
				\int_{\substack{|y|>r/2\\|x-y|<1}}
				\frac{\rho_\gamma(y)}{|x-y|^2}\,dy
				\le Ce^{-\sigma r},
				\]
				while
				\[
				\int_{\substack{|y|>r/2\\|x-y|\ge1}}
				\frac{\rho_\gamma(y)}{|x-y|^2}\,dy
				\le
				\int_{|y|>r/2}\rho_\gamma(y)\,dy
				\le Ce^{-\sigma r}.
				\]
				This proves \eqref{eq:W2-rate}.
				
				For $\alpha=2$, equation \eqref{eq:effectiveA} becomes
				\[
				(-\Delta+\kappa_i^2)u_i=2W_2u_i.
				\]
				Kato's inequality and \eqref{eq:W2-rate} imply that, for some $C_0>0$,
				\[
				(-\Delta+\kappa_i^2)|u_i|
				\le
				\frac{C_0}{|x|^2}|u_i|
				\qquad\text{for }|x|\gg1.
				\]
				Let $r=|x|$ and define
				\[
				\Phi_a(r):=\frac1r\exp\!\left(-\kappa_i r-\frac{a}{r}\right),
				\qquad a>0.
				\]
				A direct computation gives
				\[
				\frac{(-\Delta+\kappa_i^2)\Phi_a}{\Phi_a}
				=
				\frac{2\kappa_i a}{r^2}
				+\frac{2a}{r^3}
				-\frac{a^2}{r^4}.
				\]
				Choose $a>0$ so large that $2\kappa_i a>C_0$. Then, for all
				sufficiently large $r$,
				\[
				\left(
				-\Delta+\kappa_i^2-\frac{C_0}{r^2}
				\right)\Phi_a\ge0.
				\]
				The exterior comparison principle therefore yields
				\[
				|u_i(x)|
				\le
				C\Phi_a(|x|)
				\le
				C_i|x|^{-1}e^{-\kappa_i|x|}.
				\]
				Applying the interior elliptic gradient estimate on $B_1(x)$ and using
				the boundedness of $W_2$ gives the corresponding estimate for
				$|\nabla u_i(x)|$. This proves \eqref{eq:sharp-upper}.
				
				Finally, let $u_1>0$ be the first eigenfunction and
				$\kappa_1=\sqrt{-\mu_1}$. Since
				\[
				(-\Delta+\kappa_1^2)u_1=2W_2u_1
				\]
				and $2W_2u_1\ge0$ is nonzero, the resolvent representation gives
				\[
				u_1(x)
				=
				2\int_{\R^3}
				G_{\kappa_1}(x-y)W_2(y)u_1(y)\,dy,
				\qquad
				G_{\kappa}(z)=\frac{e^{-\kappa|z|}}{4\pi|z|}.
				\]
				Choose $R>0$ such that
				\[
				\int_{B_R}W_2(y)u_1(y)\,dy>0.
				\]
				For $|x|>2R$ and $y\in B_R$,
				\[
				|x-y|\le |x|+R,
				\]
				and therefore
				\[
				G_{\kappa_1}(x-y)
				\ge
				\frac{e^{-\kappa_1(|x|+R)}}{4\pi(|x|+R)}
				\ge
				c\,|x|^{-1}e^{-\kappa_1|x|}.
				\]
				Restricting the resolvent integral to $B_R$ proves
				\eqref{eq:sharp-lower}.
			\end{proof}

			\section{A dual Hartree-type Lieb--Thirring inequality}\label{sec:B}
			\setcounter{theorem}{0}
			\renewcommand{\thetheorem}{\Alph{section}.\arabic{theorem}}
			For completeness, we record in this appendix a dual inequality associated with the finite-rank Hartree GNS inequality established in Section~\ref{sec:2}.
			More precisely, letting $q'=q/(q-1)$, with the convention $q'=1$ when $q=\infty$, we obtain the following Hartree-type Lieb--Thirring inequality.

			\begin{theorem}[Hartree-type Lieb--Thirring inequality]\label{thm:LT}
				Let $1\le N<\infty$, $0<\alpha<3$, and $1<q\le\infty$. Assume
				\[
				\int_{\R^3}(|V|*|x|^{-\alpha})|V|\,dx<\infty.
				\]
				Then
				\begin{equation}
					\sum_{n=1}^{N}\left|\lambda_n\bigl(-\Delta+V*|x|^{-\alpha}\bigr)\right|^{q'}
					\le L_{\alpha,q'}^{(N)}
					\left(\int_{\R^3}(V_-*|x|^{-\alpha})V_-\,dx\right)^{\frac{2q'}{4-\alpha}},
					\label{eq:LT}
				\end{equation}
				where $L_{\alpha,q'}^{(N)}$ is the optimal constant and missing negative eigenvalues are counted as zero. Moreover,
				\begin{equation}
					\mathcal{K}_{\alpha,q}^{(N)}\bigl(L_{\alpha,q'}^{(N)}\bigr)^{\frac{(4-\alpha)(q-1)}{\alpha q}}
					=\frac{\alpha}{4}\left(\frac{4-\alpha}{4}\right)^{\frac{4-\alpha}{\alpha}}.
					\label{eq:duality}
				\end{equation}
			\end{theorem}
			We next prove Theorem~\ref{thm:LT} and show the dual relation between non-local Gagliardo-Nirenberg-Sobolev inequality and Lieb-Thirring type inequality, which inspired   by  the arguments in Appendix A of \cite{FGL}.

			\begin{proof}[Proof of Theorem~\ref{thm:LT}]
				Define an operator 
				$$\gamma=\sum_{j=1}^N n_j|u_j\rangle \langle u_j|,\ \text{ with }\ n_j=|\lambda_j(-\Delta+V*|x|^{-\alpha})|^{\frac{1}{q-1}}, \ j=1,2,\cdots,N, $$ 
				where we assume that $-\Delta+V*|x|^{-\alpha}$ has at least $N$ negative eigenvalues (counting multiplicity), and $u_1,u_2,...,u_N$ be orthogonal eigenfunctions corresponding to the $N$ negative eigenvalues.
				Set $H_{\alpha}:=\int_{\mathbb{R}^3}(\rho_\gamma*|x|^{-\alpha})\rho_\gamma dx$, it then follows that
				\begin{equation*}
					\begin{aligned}
						\sum_{j=1}^N &|\lambda_j(-\Delta+V(x)*|x|^{-\alpha})|^{q'}=-\sum_{j=1}^N n_j\int_{\mathbb{R}^3}\Big( |\nabla u_j|^2+(V(x)*|x|^{-\alpha})|u_j|^2\Big)dx\\
						\le&-\mathcal{K}_{\alpha,q}^{(N)}\|\gamma\|_{\mathcal{S}^{q}}^{-\frac{4-\alpha}{\alpha}}H_{\alpha}^{\frac{2}{\alpha}}
						+\int_{\mathbb{R}^3}(\rho_\gamma*|x|^{-\alpha})V_-(x)dx\\
						\le&-\mathcal{K}_{\alpha,q}^{(N)}\|\gamma\|_{\mathcal{S}^{q}}^{-\frac{4-\alpha}{\alpha}}H_{\alpha}^{\frac{2}{\alpha}}+H_\alpha^{\frac{1}{2}}\Big(\int_{\mathbb{R}^3}(V_-*|x|^{-\alpha})V_-(x)dx\Big)^{\frac{1}{2}}.
					\end{aligned}
				\end{equation*}
				Treat $H_\alpha$ as a variable and find its minimum value, we have
				\begin{equation*}
					\sum_{j=1}^N |\lambda_j(-\Delta+V(x)*|x|^{-\alpha})|^{q'}\le \frac{4-\alpha}{4}\Big(\frac{\alpha}{4}\Big)^{\frac{\alpha}{4-\alpha}}(\mathcal{K}_{\alpha,q}^{(N)})^{-\frac{\alpha}{4-\alpha}}\|\gamma\|_{\mathcal{S}^{q}}\Big(\int_{\mathbb{R}^3}(V_-*|x|^{-\alpha})V_-(x)dx\Big)^{\frac{2}{4-\alpha}}.
				\end{equation*}
				It can deduce from $\|\gamma\|_{\mathcal{S}^{q}}^{q}=\sum_{j=1}^N |\lambda_j(-\Delta+V(x)*|x|^{-\alpha})|^{q'}$ that
				\begin{equation*}
					\sum_{j=1}^N |\lambda_j(-\Delta+V(x)*|x|^{-\alpha})|^{q'}\le \Big(\frac{4-\alpha}{4}\Big)^{q'}\Big(\frac{\alpha}{4}\Big)^{\frac{\alpha q'}{4-\alpha}} (\mathcal{K}_{\alpha,q}^{(N)})^{-\frac{\alpha q'}{4-\alpha}}\Big(\int_{\mathbb{R}^3}(V_-*|x|^{-\alpha})V_-(x)dx\Big)^\frac{2q'}{4-\alpha}.
				\end{equation*}
				Which indicates that the Hartree-type Lieb--Thirring inequality \eqref{eq:LT} is well defined, and   
				\begin{equation}\label{dual}
					\Big(\frac{4-\alpha}{4}\Big)^{\frac{q}{q-1}}\Big(\frac{\alpha}{4}\Big)^{\frac{\alpha q}{(4-\alpha)(q-1)}} (\mathcal{K}_{\alpha,q}^{(N)})^{-\frac{\alpha q}{(4-\alpha)(q-1)}}\ge \mathcal{L}_{\alpha,q'}^{(N)}.
				\end{equation}

				On the other hand, for any given  operator $0\leq \gamma=\sum_{i=1}^Nn_i|u_i\rangle \langle u_i|$,  
				we choose $V(x)=-\beta \rho_\gamma$ with $\beta>0$ to be determined. Then
				\begin{equation*}
					\begin{aligned}
						&\sum_{j=1}^N n_j\int_{\mathbb{R}^3}|\nabla u_j|^2dx-\beta\int_{\mathbb{R}^3}(\rho_\gamma*|x|^{-\alpha})\rho_\gamma dx
						=\sum_{j=1}^N n_j \int_{\mathbb{R}^3}\Big(|\nabla u_j|^2dx-\beta(\rho_\gamma*|x|^{-\alpha})|u_j|^2\Big) dx\\
						&\ge -\|\gamma\|_{\mathcal{S}^{q}}\Big(\sum_{j=1}^N |\lambda_j(-\Delta-\beta\rho_\gamma*|x|^{-\alpha})|^{q'}\Big)^{\frac{1}{q'}}
						\ge -\|\gamma\|_{\mathcal{S}^{q}}(\mathcal{L}_{\alpha,q'}^{(N)})^{\frac{1}{q'}}\beta^{\frac{4}{4-\alpha}}\Big(\int_{\mathbb{R}^3}(\rho_\gamma*|x|^{-\alpha})\rho_\gamma dx\Big)^{\frac{2}{4-\alpha}},
					\end{aligned}
				\end{equation*}
				where we used the H\"older inequality in Schatten space \cite{Simon}. Hence
				\begin{equation*}
					\sum_{j=1}^N n_j\int_{\mathbb{R}^3}|\nabla u_j|^2dx\ge \beta H_{\alpha}-\|\gamma\|_{\mathcal{S}^{q}}(\mathcal{L}_{\alpha,q'}^{(N)})^{\frac{1}{q'}}\beta^{\frac{4}{4-\alpha}}H_{\alpha}^{\frac{2}{4-\alpha}}.
				\end{equation*}
				Optimizing over $\beta$, we obtain
				\begin{equation*}
					\sum_{j=1}^N n_j\int_{\mathbb{R}^3}|\nabla u_j|^2dx\ge \Big(\frac{4-\alpha}{4}\Big)^{\frac{4-\alpha}{\alpha}}\Big(\frac{\alpha}{4}\Big) (\mathcal{L}_{\alpha,q'}^{(N)})^{-\frac{4-\alpha}{\alpha q'}} \|\gamma\|_{\mathcal{S}^{q}}^{-\frac{4-\alpha}{\alpha}}H_{\alpha}^{\frac{2}{\alpha}}.
				\end{equation*}
				The GNS inequality \eqref{eq:GNS} gives
				\begin{equation*}
					\frac{\sum_{j=1}^N n_j\int_{\mathbb{R}^3}|\nabla u_j|^2dx \|\gamma\|_{\mathcal{S}^{q}}^{\frac{4-\alpha}{\alpha}} }{H_{\alpha}^{\frac{2}{\alpha}}}\ge \mathcal{K}_{\alpha,q}^{(N)}\ge \Big(\frac{4-\alpha}{4}\Big)^{\frac{4-\alpha}{\alpha}}\Big(\frac{\alpha}{4}\Big) (\mathcal{L}_{\alpha,q'}^{(N)})^{-\frac{(4-\alpha)(q-1)}{\alpha q}}.
				\end{equation*}
				Combining this with \eqref{dual} yields
				\begin{equation*}
					\mathcal{K}_{\alpha, q}^{(N)}(\mathcal{L}_{\alpha,q'}^{(N)})^{\frac{(4-\alpha)(q-1)}{\alpha q}}=\frac{\alpha}{4}\Big(\frac{4-\alpha}{4}\Big)^{\frac{4-\alpha}{\alpha}}.
				\end{equation*}
				This proves \eqref{eq:duality} and completes the proof.
			\end{proof}
			
			\noindent{\bf Acknowledgments.} This work was supported by the National Natural Science Foundation of China under grant numbers 12322106, 12601205, 12671142 and supported by   ``the Fundamental Research Funds for the Central Universities”.
			
			\vskip.1truein 

\noindent\textbf{Data availability} Data sharing is not applicable to this article as no datasets were generated or analyzed during the current study.
\vskip.1truein
\noindent\textbf{Conflict of Interest: }  The authors declare that there is no conflict of interest.


\begin{thebibliography}{GNN}
				\bibitem{BrackMurthy}Brack, M. and Murthy, M. V. N. Harmonically trapped fermion gases: exact and asymptotic results in arbitrary dimensions, J. Phys. A: Math. Gen. {\bf36} (2003), 1111-1133.
				\bibitem{ChenGuo} Chen, B., Guo, Y. Ground states of Fermionic nonlinear Schr\"odinger systems with Coulomb potential I: the $L^2$-subcritical case. Lett. Math. Phys. 114 (2024), no. 6, Paper No. 132.
				\bibitem{ChenGuoNamOuYang}Chen, B., Guo, Y., Nam, P., Ou Yang, D. Mass-Critical Neutron Stars in the Hartree-Fock and Hartree-Fock-Bogoliubov Theories. arXiv:2504.16462.
				\bibitem{Cycon} Cycon, H. L., Froese, R. G., Kirsch, W. and Simon, B. Schr\"odinger Operators with Application to Quantum Mechanics and Global Geometry, Texts and Monographs in Physics, Springer Study Edition, Springer-Verlag, Berlin, 1987.
				\bibitem{Dean}Dean, D. S., Doussal, P. L., Majumdar S. N. and Schehr, G. Noninteracting fermions in a trap and random matrix theory, J. Phys. A: Math. Theor. 52 (2019), 144006, 32 pp.
				\bibitem{FGL}Frank, R., Gontier, D., Lewin, M. The nonlinear Schr\"odinger equation for orthonormal functions II: Application to Lieb-Thirring inequalities. Comm. Math. Phys. 384 (2021), no.3, 1783-1828.
				\bibitem{pde}Gilbarg, D., and Trudinger, N.~S.  Elliptic partial differential equations of second order, second edition, Grundlehren der mathematischen Wissenschaften, 224, Springer, Berlin, 1983.
				\bibitem{Giorgini} Giorgini, S. Pitaevskii, L. P. and Stringari, S.  Theory of ultracold atomic Fermi gases, Rev. Mod. Phys. 80 (2008), 1215-1274.
				\bibitem{Lewin} Gontier, D. Lewin, M. and Nazar, F. Q. The nonlinear Schr\"odinger equation for orthonormal functions: existence of ground states, Arch. Ration. Mech. Anal. 240 (2021), 1203-1254.
				\bibitem{GuoLi}Guo, Y., Li, Y. Mass concentration of two-spinless Fermi systems with attractive interactions. J. Differential Equations 410 (2024), 197-227.	
				\bibitem{GLW}Guo, Y., Li, Y. and Wu, S. Ground States of Attractive Fermi Schr\"odinger Systems with Ring-Shaped Potentials. Calc. Var. 65, 230 (2026). 		
				
				\bibitem{GuoSeiringer} Guo, Y., Seiringer, R. On the mass concentration for Bose-Einstein condensates with attractive interactions. Lett. Math. Phys. 104 (2014), no.2, 141-156.
				\bibitem{GuoWangZengZhou}Guo, Y., Wang, Z., Zeng, X., Zhou, H. Properties of ground states of attractive Gross-Pitaevskii equations with multi-well potentials. Nonlinearity 31 (2018), no.3, 957-979.
				\bibitem{GuoZengZhou} Guo, Y., Zeng, X., Zhou, H. Energy estimates and symmetry breaking in attractive Bose-Einstein condensates with ring-shaped potentials. Ann. Inst. H. Poincar\'{e} C Anal. Non Lin\'{e}aire 33(2016), no.3, 809-828.
				\bibitem{Hof} Hoffmann-Ostenhof, M. and Hoffmann-Ostenhof, T. Schr\"odinger inequalities and asymptotic behavior of the electron density of atoms and molecules, Phys. Rev. A 16 (1977), 1782-1785.
				\bibitem{LenzmannLewin}Lenzmann, E., Lewin, M. Minimizers for the Hartree-Fock-Bogoliubov theory of neutron stars and white dwarfs. Duke Math. J. 152 (2010), no.2, 257-315.
				\bibitem{Lieb}Lieb, E.H. and Loss, M. Analysis, Graduate Studies in Mathematics, Vol. 14, 2nd edn. Amer. Math. Soc, Providence, RI 2001.
				\bibitem{jfa2013}Moroz, V. and Van~Schaftingen, J. Groundstates of nonlinear Choquard equations: existence, qualitative properties and decay asymptotics, J. Funct. Anal. 265 (2013), no.~2, 153-184
				\bibitem{Reed}Reed, M.~C. and Simon, B. Methods of modern mathematical physics. II. Fourier analysis, self-adjointness, Academic Press, New York-London, 1975
				\bibitem{Reed1}Reed, M.~C. and Simon, B. Methods of Modern Mathematical Physics IV. Analysis of Operators, Academic Press, New York, 1978.
				\bibitem{Simon} Simon, B. Trace Ideals and Their Applications, Mathematical Surveys and Monographs, Vol. 120, 2nd edn. American Mathematical Society, Providence, 2005.
				\bibitem{Smith}Smith, N. R., Dean, D. S., Doussal, P. L., Majumdar S. N. and Schehr, G. Noninteracting trapped fermions in double-well potentials: inverted-parabola kernel, Phys. Rev. A 101 (2020), 053602, 14 pp.
				\bibitem{Sol}Solovej, J. P. Proof of the ionization conjecture in a reduced Hartree-Fock model, Invent. Math. 104 (1991), 291-311.
				\bibitem{Willem} Willem, M. Minimax Theorems, Progress in Nonlinear Differential Equations and Their Applications, Vol. 24, Birkh\"auser Boston, Inc. Boston, 1996.
				\bibitem{WuZengZhang}Wu, Y., Zeng, X. and Zhang, Y. Gagliardo-nirenberg-sobolev inequalities and	ground states of fermions in relativistic hartree-fock model. arxiv:2505.06613.
				
				
				
			\end{thebibliography}
		\end{document}